\documentclass[11pt]{article}
\usepackage{amsmath, amssymb, theorem, latexsym, epsfig, faktor, xcolor}
\usepackage{eufrak}
\numberwithin{equation}{section}
                 
\theoremstyle{plain}
\theorembodyfont{\itshape}
\newtheorem{theorem}{Theorem}[section]
\newtheorem{proposition}[theorem]{Proposition}
\newtheorem{lemma}[theorem]{Lemma}
\newtheorem{corollary}[theorem]{Corollary}

\theorembodyfont{\rmfamily}
\newtheorem{definition}[theorem]{Definition}

\newtheorem{remark}[theorem]{Remark}
\newenvironment{proof}{{\noindent \textbf{Proof}\,\,}}{\hspace*{\fill}$\Box$\medskip}

\graphicspath{{figures/}}

\def\qq{\mathbb Q}

\def\wt#1{\widetilde#1}
\def\wh#1{\widehat#1}
\def\rr{\mathbb R}
\def\var{\varepsilon}
\def\tt{\mathbb T}
\def\var{\varepsilon}

\def\mca{\mathcal A}
\def\mcb{\mathcal B}

\def\mo{\operatorname{mod}}

\def\nn{\mathbb N}

\def\zz{\mathbb Z}

\def\mod{\operatorname{mod}}
\def\Id{\operatorname{Id}}
\def\diam{\operatorname{diam}}
\def\const{\operatorname{const}}
\def\R{\mathbb R}
\def\T{\mathbb T}
\def\Q{\mathbb Q}
\def\Z{\mathbb Z}
\def\N{\mathbb N}
\def\spann{\operatorname{span}}
\def\SO{\operatorname{SO}}
\def\SU{\operatorname{SU}}
\def\SL{\operatorname{SL}}

\def\a{\alpha}
\def\lam{\lambda}
\def\d{\delta}
\def\D{\Delta}
\def\f{\varphi}
\def\om{\omega}
\def\tb{\bar{t}}
\def\that{\widehat{t}}
\def\tig{\widetilde{g}}
\def\cut{\operatorname{cut}}
\def\tcut{t_{\cut}}
\def\ip{\langle \cdot, \cdot \rangle}
\def\bg{\bar{g}}
\def\ba{\bar{a}}
\def\bb{\bar{b}}
\def\bc{\bar{c}}
\def\tg{\tilde{g}}
\def\ta{\tilde{a}}
\def\tib{\tilde{b}}
\def\tc{\tilde{c}}
\def\tf{\tilde{f}}
\def\tit{\tilde{t}}
\def\tK{\tilde{K}}
\def\hK{\hat{K}}
\def\ud{\underline{d}}
\def\od{\overline{d}}
\def\la{\lambda}

\def\ts{t_*}
\def\tsd{t_*^2}
\def\tsf{t_*^4}
\def\ps{p_*}

\def\vH{\vec H}

\def\sR{sub-Rie\-man\-ni\-an }

\newcommand{\restr}[2]{\left. #1 \right|_{#2}}
\newcommand{\be}[1]{\begin{equation}\label{#1}}
\newcommand{\ee}{\end{equation}}
\newcommand{\eq}[1]{$(\protect\ref{#1})$}
\newcommand{\map}[3]{#1 \, : \, #2 \to #3}
\newcommand{\mapto}[3]{#1 \, : \, #2 \mapsto #3}
\newcommand{\pder}[2]{\frac{\partial \, #1}{\partial \, #2} }

\newcommand{\onefiglabelsizen}[4]
{
\begin{figure}[htbp]
\begin{center}
\includegraphics[height=#4cm]{#1}
\\
\parbox[t]{0.7\textwidth}{\caption{#2}\label{#3}}
\end{center}
\end{figure}
}

\newcommand{\twofiglabelsize}[8]
{
\begin{figure}[htbp]
\includegraphics[height=#4\textwidth]{#1}
\hfill
\includegraphics[height=#8\textwidth]{#5}
\\
\hfill
\parbox[t]{0.45\textwidth}{\caption{#2}\label{#3}}
\hfill
\parbox[t]{0.45\textwidth}{\caption{#6}\label{#7}}
\hfill
\end{figure}
}

\title{Sub-Riemannian geodesics \\on the Heisenberg 3D nil-manifold}
\author{A. Glutsyuk\thanks{UMPA, ENS de Lyon (UMR 5669 du CNRS), France. E-mail: 
aglutsyu@ens-lyon.fr} \thanks{HSE University, Moscow, Russia} \thanks{Higher School of Modern Mathematics MIPT, 1 Klimentovskiy per., Moscow, Russia}, Yu. Sachkov\thanks{Program Systems Institute (Russian Academy of Sciences, Pereslavl-Zalessky, Russia), RUDN University (Moscow, Russia)} \thanks{Corresponding author, yusachkov@gmail.com}}
\begin{document}
\maketitle
\newcommand{\bigslant}[2]{{\raisebox{-.2em}{$#1$}\left\backslash\raisebox{.2em}{$#2$}\right.}}

\begin{abstract}
We study the projection of the left-invariant \sR structure on the 3D Heisenberg group $G$ to the Heisenberg 3D nil-manifold~$M$ --- the compact homogeneous space of $G$ by the discrete Heisenberg group. 

First we describe dynamical properties of the geodesic flow for $M$: periodic and dense orbits, a dynamical characterization of the normal Hamiltonian flow of Pontryagin maximum principle and its integrability properties. 
We show that it is Liouville integrable on a nonzero level hypersurface $\Sigma$ of the Hamiltonian  outside appropriate smaller  proper hypersurface in $\Sigma$ and has no nontrivial analytic integrals 
on all of $\Sigma$. Then we obtain sharp twoside bounds of \sR balls and distance in~$G$, and on this basis we estimate the cut time for \sR geodesics in $M$. 
\end{abstract}



\tableofcontents

\section{Introduction}\label{sec:intro}
The left-invariant \sR structure on the 3D Heisenberg group $G$ is a paradigmatic model of \sR geometry \cite{montgomery_book, ABB_book}. In this paper we study the projection of this \sR structure to a compact homogeneous space of the group $G$ --- to the Heisenberg 3D nil-manifold $M$. The \sR structure on $M$ is locally isometric to the structure on~$G$, thus these structures have local objects (geodesics and conjugate points) related by the projection. 
Although, the global issues as dynamical properties of geodesics and cut time are naturally different. 
We aim to study these global questions in some detail.

The structure of this work is as follows. 
Section \ref{sec:ex} discusses a well known projection of Euclidean structure from $\R^n$ to the torus $\T^n$, which suggests a motivation of the subsequent study.
In Sec. \ref{sec:Heis}  we recall the construction of the Heisenberg 3D nil-manifold $M$. In Sec. \ref{sec:SR} we present basic definitions of \sR geometry, define the \sR structures of $G$ and~$M$, and describe their geodesics; in particular, we recall the parametrization of two distinct classes of \sR geodesics in $G$ --- lines and spirals. In Sec. \ref{sec:lines} we show that \sR geodesics-lines in $M$ may be either closed or dense, and describe explicitly geodesics falling into these classes.
In Sec. \ref{sec:spirals} we describe dynamical properties of geodesics-spirals in $M$: we show that such a geodesic is either closed or dense in a certain 2D torus, and distinguish geodesics of these classes. In Sec.   \ref{sec:dyn} we describe dynamical properties of the restriction of the Hamiltonian vector field for geodesics to a compact invariant surface (common level surface of the Hamiltonian and the Casimir). We show that the flow of this restriction is conjugated to a $p$-standard flow. Then we show that the restriction of the flow to 
a nonzero level hypersurface $\Sigma$ of the Hamiltonian is Liouville integrable outside appropriate hypersurface $S_0$ and has no nontrivial analytic integrals 
on the whole hypersurface $\Sigma$. Next in Sec. \ref{sec:ellips} we obtain sharp interior and exterior ellipsoidal bounds of \sR balls in $G$, which improve the classical ball-box bounds. In Sec. \ref{sec:lower} we estimate the cut time on geodesics in $M$ and the diameter of $M$ on the basis of above interior bounds of \sR balls in $G$. Finally, in Sec. \ref{sec:upper} we prove two-sided bounds of cut time on geodesics in $M$   on the basis of above exterior bounds of \sR balls in $G$.

{Conserning related research on sub-Riemannian optimal control problems on compact homogeneous spaces, we are aware only on the works where left-invariant sub-Riemannian structures on $\SO(3)$ and $\SU(2)$ were studied~\cite{VG, boscain_rossi, ber_zub01, ber_zub15, ber_zub16, 
chang_markina_vasiliev}. Moreover, sub-Riemannian structures on the lens spaces were studied in the paper  \cite{boscain_rossi}.
As far as we know, our paper is the first one in the literature where dynamical properties of sub-Riemannian geodesic flow on a compact homogeneous spase are investigated. 
}

\section{Motivating example}\label{sec:ex}

Geodesics in the Euclidean space $\R^n$ have trivial dynamics (they tend to infinity) and optimality properties (they are length minimizers).  

The situation changes when we pass from $\R^n$  to its compact homogeneous space --- the torus $\T^n = \R^n /\Z^n$. Consider the Riemannian structure on $\T^n$ obtained via the projection $\map{\pi}{\R^n}{\T^n}$. Then the geodesics on $\T^n$ are orbits of the linear flows 
\begin{align}
&t \mapsto (x_1^0 + \om_1 t, \dots, x_n^0 + \om_n t)(\mod 1), \label{geodRn}\\
&\om_1^2 + \dots +\om_n^2 \neq 0, \qquad (x_1^0, \dots, x_n^0) \in \T^n, \qquad t \in \R. \nonumber
\end{align}

Kronecker's theorem \cite{kathas} (Propos. 1.5.1) states that such a geodesic is dense in $\T^n$ if and only if the frequencies $\om_1, \dots, \om_n$ are linearly independent over $\Q$. In all other cases a geodesic is dense in a nontrivial $k$-dimensional torus in $\T^n$, $1 \leq k < n$, see Propos. \ref{prop:linflow} below; in particular, a geodesic is periodic if $k = 1$. See Figs. \ref{fig:geod11}, \ref{fig:geod12} for $n = 2$ and Figs.~\ref{fig:geod13}--\ref{fig:geod15} for $n = 3$.

The following well-known statement generalizes Kronecker's theorem.

\begin{proposition} (see \cite[section 5.1.5]{kh2}). 
\label{prop:linflow}
Consider a geodesic $\Gamma$ in $\T^n$ of the form \eq{geodRn} and the corresponding vector space
$$
R = \spann_{\Q} \left\{ r = (r_1, \dots, r_n) \in \Q^n \mid \sum_{i=1}^n r_i \om_i = 0\right\}, \qquad \rho = \dim R.
$$
Then $\Gamma$ is dense in a smooth manifold $S \subset \T^n$ diffeomorphic to a torus~$\T^{n- \rho}$. 
\end{proposition}

\twofiglabelsize
{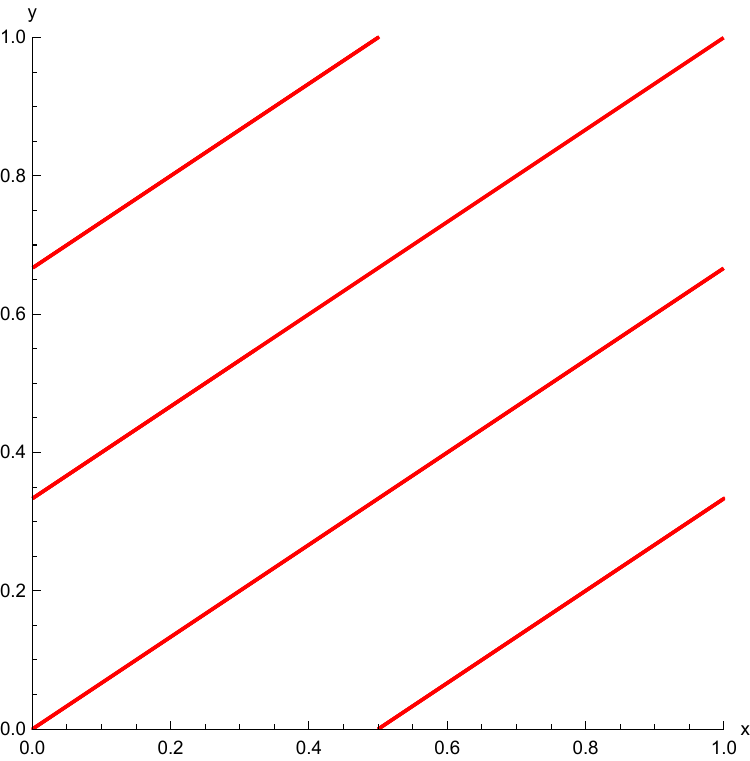}{Geodesic \eq{geodRn}  for $n = 2$, $(\om_1,\om_2) = (3,2)$}{fig:geod11}{0.5}
{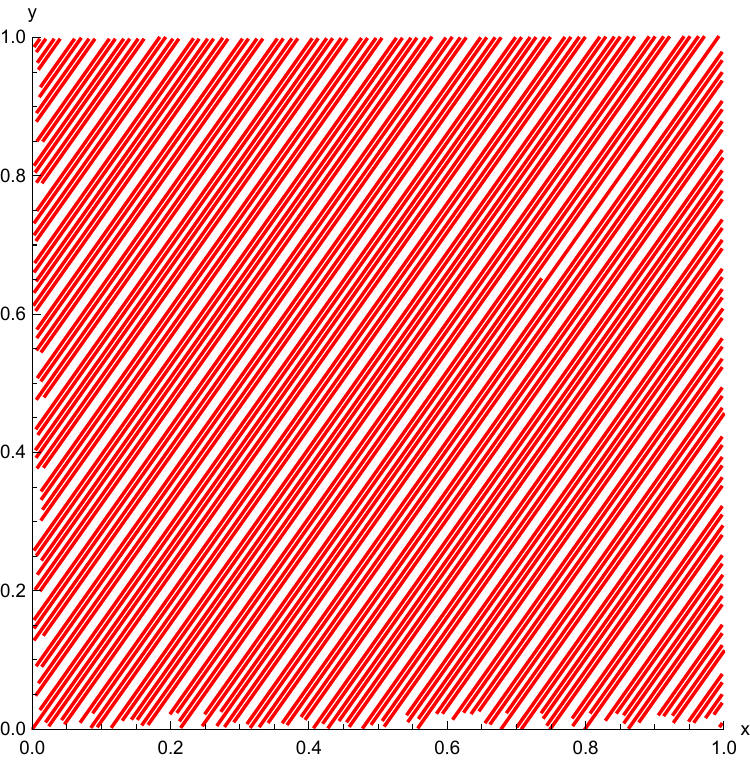}{Geodesic \eq{geodRn}  for $n = 2$, $(\om_1,\om_2) = (1,\sqrt{2})$}{fig:geod12}{0.5}

\twofiglabelsize
{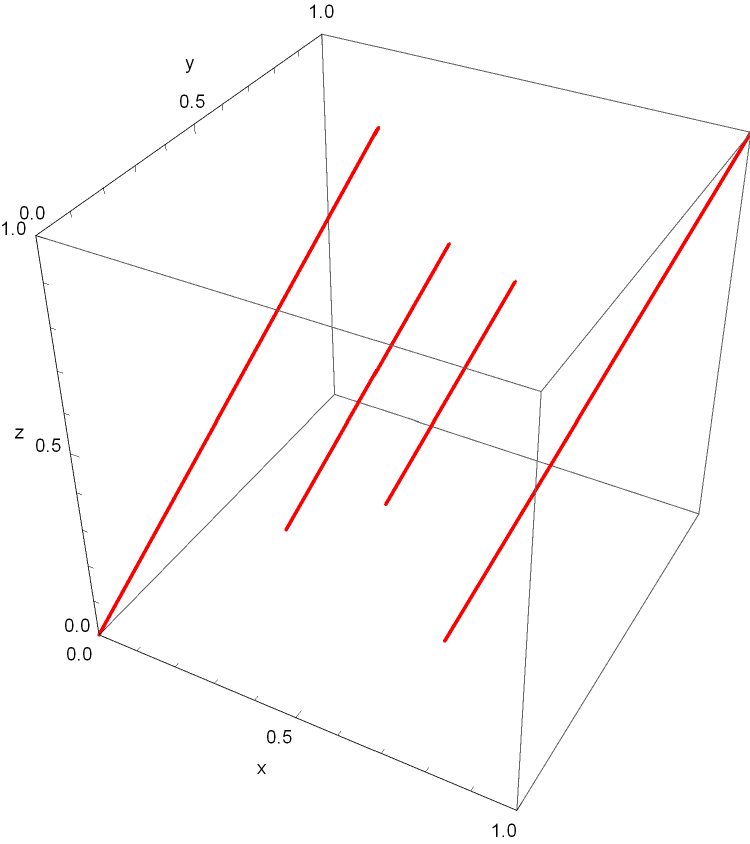}{Geodesic \eq{geodRn}  for $n = 3$, $(\om_1, \om_2, \om_3) = (1, 2, 3)$}{fig:geod13}{0.5}
{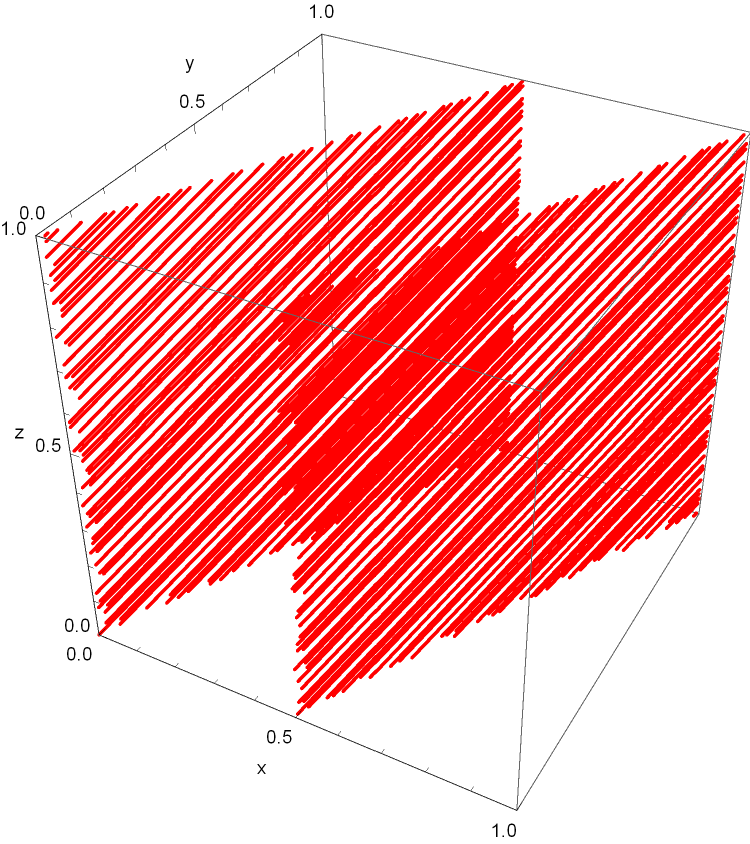}{Geodesic \eq{geodRn}  for $n = 3$, $(\om_1, \om_2, \om_3) = (1, 2, \sqrt{2})$}{fig:geod14}{0.5}

\onefiglabelsizen
{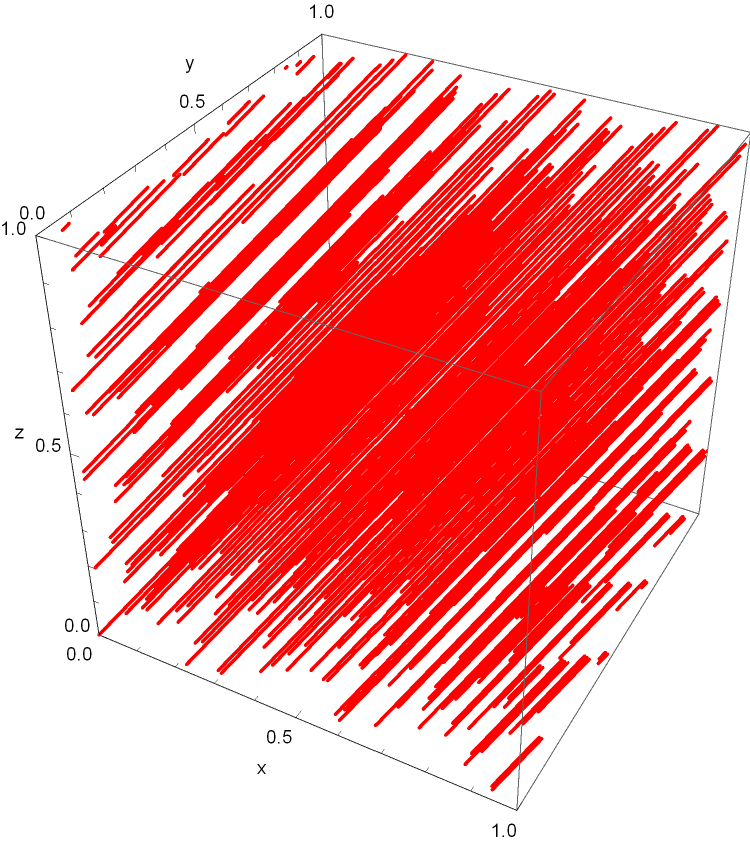}{Geodesic \eq{geodRn}  for $n = 3$, $(\om_1, \om_2, \om_3) = (1, \sqrt{3}, \sqrt{2})$}{fig:geod15}{6}

Moreover, each geodesic \eq{geodRn} loses optimality at an instant 
\be{tcut}
\tcut = \frac{1}{2 \max |\om_i|} \in (0, + \infty).
\ee
More precisely, the {\it cut time} for a geodesic $x(t)$, $t > 0$, is defined as follows:
$$
\tcut(x(\cdot)) = \sup \{ t_1> 0 \mid x(\cdot) \text{ is length minimizing on } [0, t_1]\}.
$$
The reason for the loss of optimality is intersection with a symmetric geodesic starting from the same initial point in $\T^n$, see Fig. \ref{fig:opt_T2} for the case $n = 2$.

\onefiglabelsizen
{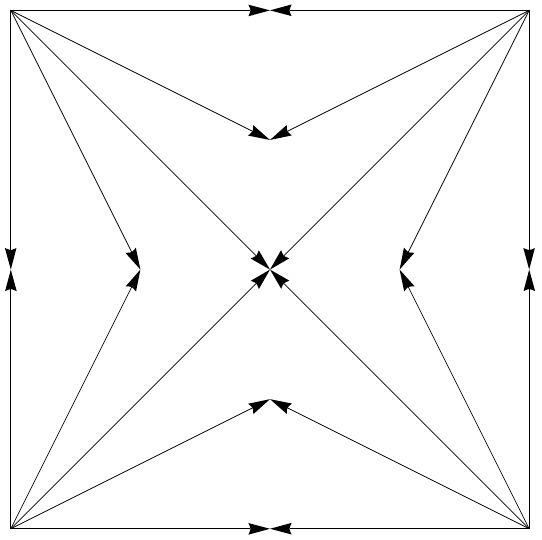}{Optimal synthesis on $\T^2$}{fig:opt_T2}{4}

In this work we aim to generalize the above projection of Euclidean structure  $\map{\pi}{\R^n}{\T^n}$ to a projection of a left-invariant \sR structure $\map{\pi}{G}{M}$ from a Lie group $G$ to its compact homogeneous space $M$. The simplest nontrivial case of such a projection is the case of the 3D Heisenberg group $G$ and the 3D Heisenberg nil-manifold $M$, see Sec. \ref{sec:Heis}.

Indeed, recall that a subgroup $H$ of a Lie group $G$ is called {\it uniform} (or {\it cocompact}) if the homogeneous space $G/H$ is compact. The only connected and simply connected non-Abelian 2D Lie group $\R \rtimes \R_+$ does not contain uniform subgroups. On the other hand, the 3D Heisenberg group $G$ has a countable number of uniform subgroups, of which the subgroup \eq{H} is the simplest one, {see discussion in Remark \ref{rem:uniform} at the end of the next section}.

\section{Heisenberg group and 3-dimensional 
nil-manifold}\label{sec:Heis}
Recall that the Heisenberg group is
$$G=\left\{\left(\begin{matrix} 1 & a & c\\ 0 & 1 & b\\ 0 & 0 & 1\end{matrix}
\right) \ | \ (a,b,c)\in\rr^3\right\}.$$
Consider the following discrete subgroup and its quotient (the {\em right} cosets space):
\be{H}H =\left\{\left(\begin{matrix} 1 & m & k\\ 0 & 1 & n\\ 0 & 0 & 1\end{matrix}
\right) \ | \ (m,n,k)\in\zz^3\right\}, \ \ \ M:= \bigslant{H}{G} = \{Hg \mid g \in G\}.\ee
The quotient $M$ is a compact smooth manifold, which is called {\it Heisenberg 3-dimensional 
nil-manifold.} 

Let $\pi:G\to M$ denote the canonical projection $g\mapsto Hg$. The functions 
$$a':=\{ a\}, \ b':=\{ b\}, \ c':=\{ c-[a]b\}$$ are coordinates on the homogeneous space $M$, 
and $\pi(a,b,c)=(a',b',c')$. 

\begin{remark} The manifold $M$ is not 
  diffeomorphic to the 3-torus $\tt^3=\rr^3_{a', b', c'}\slash\zz^3$, see \cite[section 5]{bock}. 
  Indeed, the first Betti number $b_1$ of the torus $\tt^3$ is equal to 3. On the other hand, $b_1(M)=2$. Indeed, the quotient projection $G\to M$ is a universal covering, since $G$ is diffeomorphic to $\rr^3_{a, b, c}$. Therefore, $\pi_1(M)=H$. The first homology of a path connected topological space is isomorphic to the quotient of the fundamental group by its commutant, by  classical Poincar\'e Theorem \cite[section 14.3, p.181]{ff}. 
  Hence, $H_1(M,\zz)=H\slash[H,H]$. The  commutant $[H,H]$ coincides with the subgroup 
    of integer unipotent matrices that differ from the identity just by the upper-right corner element. 
    Therefore, it is isomorphic to $\zz$. The map $H\mapsto(H_{12}, H_{23})$  is an isomorphism $H\slash[H,H]\to\zz^2$. Therefore, $b_1(M)=2$. 
\end{remark}

The manifold $M$ can be represented by a fundamental domain $D = \{(a, b, c) \mid 0 \leq a, b, c < 1\}$ with identified facets 
$\{b = 0\} \leftrightarrow \{b = 1\}$, $\{c = 0\} \leftrightarrow \{c = 1\}$, while the facets $\{a = 0\}$ and $\{a = 1\}$ are identified by the rule $(0,b,c)\simeq (1,b,c+b)$, see Fig. \ref{fig:nil_mfd}.

\onefiglabelsizen{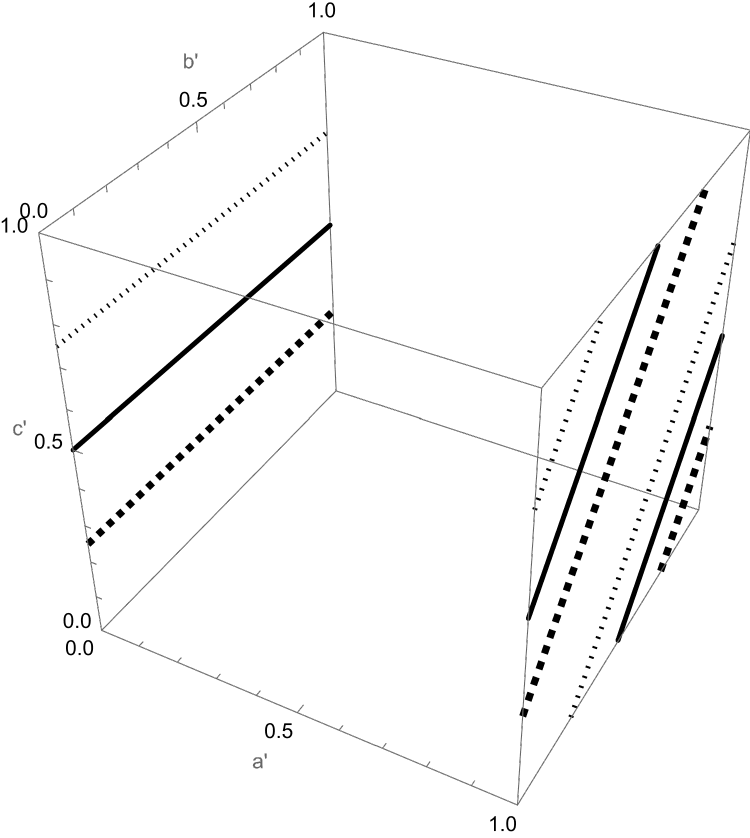}{Heisenberg 3D nil-manifold}{fig:nil_mfd}{6}

\begin{remark}\label{rem:uniform}
Any uniform subgroup of the Heisenberg group $G$ is isomorphic to a subgroup
$$
D(k) = \left\{\left(\begin{array}{ccc}
1 & a & \frac ck \\
0 & 1 & b \\
0 & 0 & 1
\end{array}\right) \mid a, b, c \in \Z \right\}
$$
for some $k \in \N$, see \cite{VGS}; in particular, $H = D(1)$. Thus any compact 3D nil-manifold is a homogeneous space $G/D(k)$. Since $G$ is simply connected, the fundamental group of such a space is $\pi_1(G/D(k)) = D(k)$.  Hence 
$$
\pi_1(M) = \pi_1(G/D(1)) = D(1) \subset D(k) = \pi_1(G/G(k)), \qquad k \in \N.
$$
 Consequently, the Heisenberg nil-manifold $M$ is the simplest compact 3D nil-manifold in the sense that it has the smallest fundamental group.  This observation motivated us to study sub-Riemannian geometry on the Heisenberg nil-manifold $M$. This work may thus  be seen as the first study of sub-Riemannian geometry of compact 3D nil-manifolds, starting from the simplest case of the Heisenberg nil-manifold $M$. We believe that our methods may be useful for the study of sub-Riemannian geometry on compact nil-manifolds of dimension~3 and greater than 3.
\end{remark}

\section{Sub-Riemannian structure on the Heisenberg group and its projection to the nil-ma\-ni\-fold}\label{sec:SR}
\subsection{Sub-Riemannian geometry}
A \sR structure \cite{montgomery_book, ABB_book} on a smooth manifold $M$ is a vector subspace distribution 
$$
\D = \{\D_q \subset T_qM \mid q \in M\} \subset TM, \qquad \dim \D_q \equiv \const, 
$$
endowed with an inner product
$$
\ip = \{\ip_q \text{ inner product in } \D_q \mid q \in M\}. 
$$
A Lipschitzian curve $\map{q}{[0, t_1]}{M}$ is called horizontal if $\dot q(t) \in \D_{q(t)}$ for almost all $t \in [0, t_1]$. The \sR length of a horizontal curve $q(\cdot)$ is
$$
l(q(\cdot)) = \int_{0}^{t_1} \langle \dot q(t), \dot q(t) \rangle^{1/2} \, dt.
$$
The \sR distance between points $q_0, q_1 \in M$ is
$$
d(q_0, q_1) = \inf \{l(q(\cdot)) \mid q(\cdot) \text{ horiz. curve s.t. } q(0) = q_0, \ q(t_1) = q_1\}.
$$
A horizontal curve is called a {\it \sR length minimizer} if its \sR length is equal to the \sR distance between its endpoints. A {\it \sR geodesic} is a horizontal curve whose sufficiently short arcs are length minimizers. Finally, a cut time along a \sR geodesic $q(\cdot)$ is
$$
\tcut(q(\cdot)) = \sup \{ \tau > 0 \mid \restr{q(\cdot)}{[0, \tau]} \text{ is a length minimizer}\}.
$$

If the distribution $\D$ is {\it completely nonholonomic} ({\it completely nonintegrable}), i.e., any points in $M$ can be connected by a horizontal curve of $\D$, then 
the \sR distance $d$ turns $M$ into a metric space, and there are naturally defined a \sR sphere of radius $R \geq 0$ centered at a point $q_0 \in M$:
$$
S_R(q_0) = \{ q \in M \mid d(q_0, q) = R\}
$$
and the corresponding \sR ball:
$$
B_R(q_0) = \{ q \in M \mid d(q_0, q) \leq  R\}.
$$

Let $X_1, \dots X_k$ be vector fields on $M$ that form an orthonormal frame of a \sR structure $(\D, \ip)$:
$$
\D_q = \spann(X_1(q), \dots, X_k(q)), \quad \langle X_i(q), X_j(q)\rangle = \d_{ij}, \qquad q \in M. 
$$
Then \sR length minimizers that connect points $q_0, q_1 \in M$ are solutions to the optimal control problem
\begin{align*}
&\dot q = \sum_{i=1}^k u_i(t) X_i(q), \qquad q \in M, \quad u = (u_1, \dots, u_k) \in \R^k, \\
&q(0) = q_0, \qquad q(t_1) = q_1, \\
&\int^{t_1}_0 \left(\sum_{i=1}^k u^2_i(t)\right)^{1/2} \, dt \to \min.
\end{align*}

\subsection{Sub-Riemannian structures on $G$ and $M$}
The left-invariant sub-Riemannian problem on the Heisenberg group
is stated as the following optimal control problem \cite{ABB_book, notes, intro}:
\begin{align}
&\dot{g} = u_1(t) X_1 (g) + u_2(t) X_2 (g), \qquad g \in G, \quad u = (u_1, u_2) \in \mathbb{R}^2, \label{SR1}\\
&g(0) = g_0 = \Id, \qquad g(t_1) = g_1, \label{SR2}\\
&\int^{t_1}_0 (u^2_1(t) + u^2_2(t))^{1/2} \, dt \to \min, \label{SR3}\\
&X_1 = \frac{\partial}{\partial x} - \frac{y}{2} \frac{\partial}{\partial z}, \quad X_2 = \frac{\partial}{\partial y} + \frac{x}{2} \frac{\partial}{\partial z}. \label{SR4}
\end{align}
The fields $X_1$, $X_2$ are left-invariant vector fields on $G$.

Here and below we use coordinates $(x, y, z)$ on the Heisenberg group $G$ such that
\be{abcxyz}
a = x, \quad b = y, \quad c = z + \frac{xy}{2}.
\ee

Geodesics for this problem have the form:
\begin{align}
&x = t\cos \theta, \label{x1}\\
&y = t\sin \theta,\label{y1}\\
&z = 0 \label{z1}
\end{align}
for $\theta \in \rr /(2 \pi \zz)$, and 
\begin{align}
&x = (\sin (\theta + h t) - \sin \theta) / h, \label{x2}\\
&y = (\cos \theta - \cos (\theta + h t)) / h, \label{y2}\\
&z = (h t - \sin h t) / (2 h^2) \label{z2}
\end{align}
for $\theta \in \rr /(2 \pi \zz)$ and $h \neq 0$.

Sub-Riemannian geodesics \eqref{x1}--\eqref{z1} are one-parametric subgroups in~$G$, they are projected to the plane $(x, y)$ into straight lines, thus we call them geodesics-lines in the sequel. Sub-Riemannian geodesics \eqref{x2}--\eqref{z2} are spirals of nonconstant slope in $\rr_{x, y, z}^3 \simeq G$, they are projected to the plane $(x, y)$ into circles, and we call them geodesics-spirals in the sequel.

Sub-Riemannian problem \eqref{SR1}--\eqref{SR4} is left-invariant on the Heisenberg group $G$, thus its projection to the nil-manifold $M$ is a well-defined sub-Riemannian problem on $M$:
\begin{align}
&\dot{g'} = u_1 X_1' (g') + u_2 X_2' (g'), \qquad g' \in M, \quad u = (u_1, u_2) \in \mathbb{R}^2, \label{SR'1}\\
&g'(0) = g_0' = \pi(\Id), \qquad g'(t_1) = g_1', \label{SR'2}\\
&\int^{t_1}_0 (u^2_1 + u^2_2)^{1/2} \, dt \to \min, \label{SR'3}\\
&X_i'(g') = \pi_*(X_i(g)), \qquad g' = \pi(g), \qquad i = 1, 2, \quad g \in G. \label{SR'4}
\end{align}
Geodesics of the projected problem \eqref{SR'1}--\eqref{SR'4} have the form $g'(t) = \pi(g(t))$, where  $g(t)$ are geodesics of the initial problem \eqref{SR1}--\eqref{SR4}.

\section[Projections of   geodesics-lines to Heisenberg nil-manifold]{Projections of   geodesics-lines \\to Heisenberg nil-manifold }\label{sec:lines} 

For every $\theta\in\rr\slash(2\pi\zz)$ consider a geodesic-line \eqref{x1}--\eqref{z1}  in $G$ 
and its projection to $M$: 
$$g(t)=(a(t), b(t), c(t))=\left( t\cos\theta, t\sin\theta, \frac{t^2}2\sin\theta\cos\theta\right), \ \ 
t\in\rr.$$
$$g'(t):= \pi\circ g(t)=(a'(t), b'(t), c'(t)),$$
\begin{equation}a'(t)=\{ t\cos\theta\},  \ b'(t)=\{ t\sin\theta\}, \ c'(t)=\left\{\frac{t^2}2\sin\theta\cos\theta-
[t\cos\theta]t\sin\theta\right\},\label{c'}\end{equation}
$$\Gamma:=\{ g'(t) \ | \ t\in\rr\}\subset M.$$
Consider the projection $\rho:M\to\tt^2=\rr^2\slash\zz^2$ and the image $\rho(\Gamma)$: 
$$\rho:(a', b', c')\mapsto(a', b'), \ \ \gamma:=\rho(\Gamma)=\{(a'(t), b'(t)) \ | \ t\in\rr\}\subset\tt^2.$$

\begin{remark} If $\theta=\frac{\pi n}2$, $n\in\zz$, then $\gamma$ and $\Gamma$ are 
both 1-periodic. If $\tan\theta\in\qq\setminus\{0\}$, then $\gamma$ is periodic, and 
$\rho^{-1}(\gamma)$ is a two-dimensional torus. 
\end{remark}

\begin{proposition} If $\tan\theta=\frac pq\in\qq\setminus\{0\}$, then the curve $\gamma$ is periodic  with period $T=\frac q{\cos\theta}$; here $(p,q)=1$. The curve $\Gamma$ is periodic 
either with the same period, as $\gamma$, if some of the numbers $p$, $q$ is even, or with twice bigger period otherwise. 
\end{proposition}

See Figs. \ref{fig:geod3}, \ref{fig:geod4}.

\twofiglabelsize{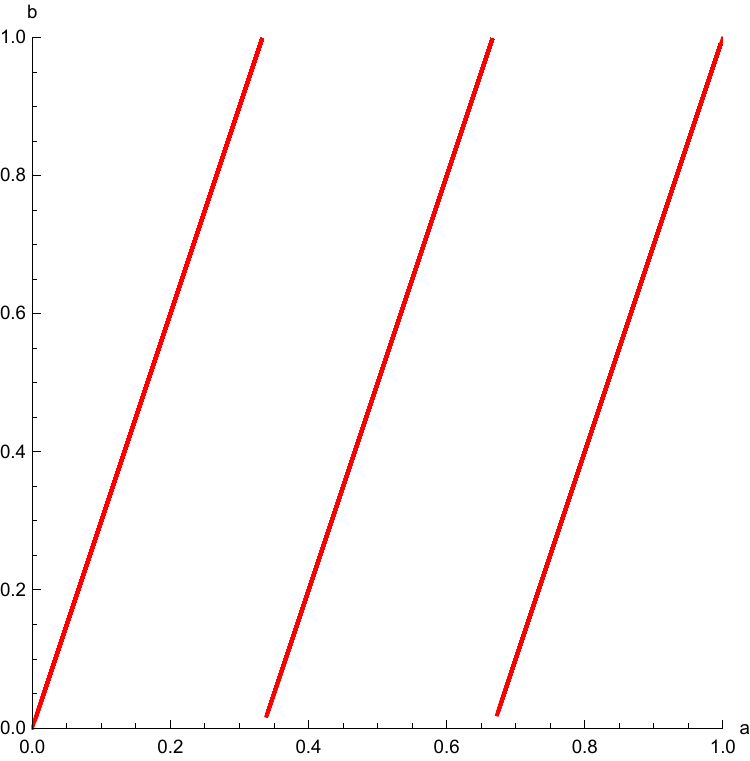}{The curve $\gamma$ for $\tan \theta = 3$}{fig:geod3}{0.5}
{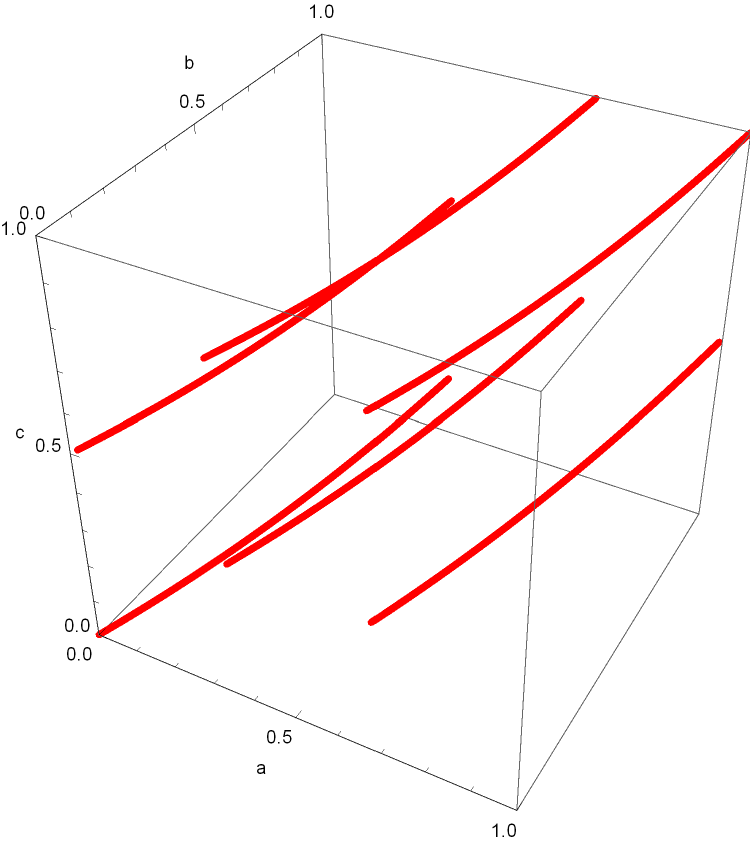}{The curve $\Gamma$ for $\tan \theta = 3$}{fig:geod4}{0.5}

\begin{proof} One has 
\begin{equation}T\cos\theta=q\in\zz, \ \ T\sin\theta=p\in\zz,\label{tsin}\end{equation}  
$$(a(t+T),b(t+T)=(a(t)+q, b(t)+p).$$
This implies $T$-periodicity of the curve $\gamma$. One has 
$$c'(t+T)=\{ c(t+T)-[a(t+T)]b(t+T)\}$$
$$=\left\{\frac{(t+T)^2}2\sin\theta\cos\theta-[(t+T)\cos\theta](t+T)\sin\theta\right\}$$
$$=\left\{\frac{t^2}2\sin\theta\cos\theta+tT\sin\theta\cos\theta+\frac{T^2}2\sin\theta\cos\theta- 
[(t+T)\cos\theta](t+T)\sin\theta\right\}.$$
Substituting (\ref{tsin}) yields 
$$c'(t+T)=\left\{\frac{t^2}2\sin\theta\cos\theta+tq\sin\theta+\frac{pq}2-[t\cos\theta+q](t\sin\theta+p)
\right\}$$
$$=\left\{\frac{t^2}2\sin\theta\cos\theta+tq\sin\theta+\frac{pq}2-([t\cos\theta]+q)(t\sin\theta+p)\right\}$$
$$=\left\{\frac{t^2}2\sin\theta\cos\theta+tq\sin\theta+\frac{pq}2-[t\cos\theta]t\sin\theta-qt\sin\theta-
p[t\cos\theta]-pq\right\}$$
$$=\left\{\frac{t^2}2\sin\theta\cos\theta-[t\cos\theta]t\sin\theta-\frac{pq}2\right\}.$$
The latter expression is equal to either $c'(t)$, if $pq$ is even, or $\{ c'(t)+\frac12\}$ otherwise. 
In the latter case replacing $T$ by $2T$ and repeating the above argument with $q$, $p$  replaced by $2q$, $2p$ yields $c'(t+2T)=c'(t)$. The proposition is proved.
\end{proof}

\begin{theorem} \label{tgdense} The curve $\Gamma$ is dense in $M$ for every 
$\theta\in\rr$ such that 
$\tan\theta\notin\qq\cup\{\infty\}$. In this case 
each its half \ $\Gamma_{\pm}=\{ g'(t) \ | \ t\in\rr_{\pm}\}$ is dense.
\end{theorem}

See Figs. \ref{fig:geod5}, \ref{fig:geod6}.

\twofiglabelsize{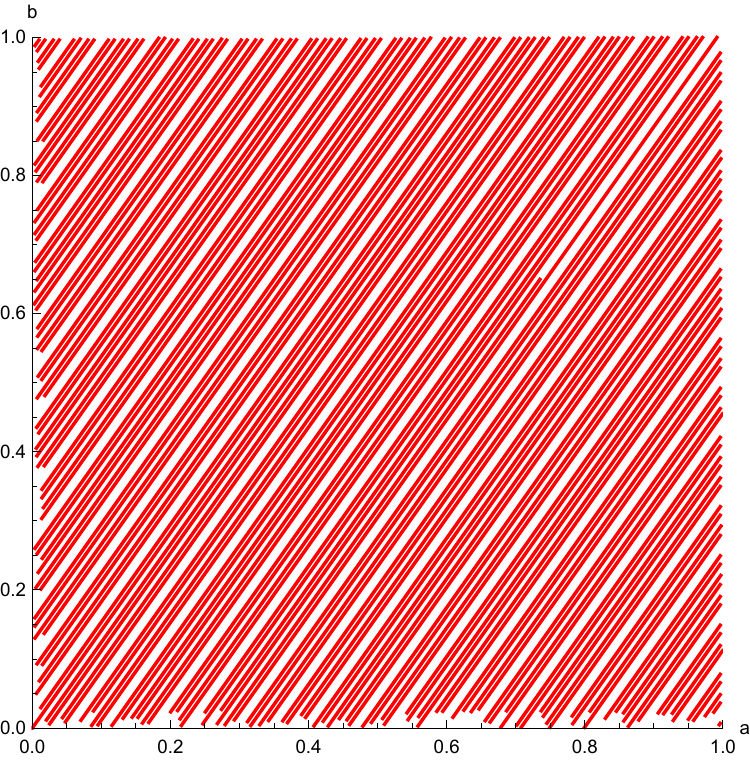}{The curve $\gamma$ for $\tan \theta = \sqrt{2}$}{fig:geod5}{0.5}
{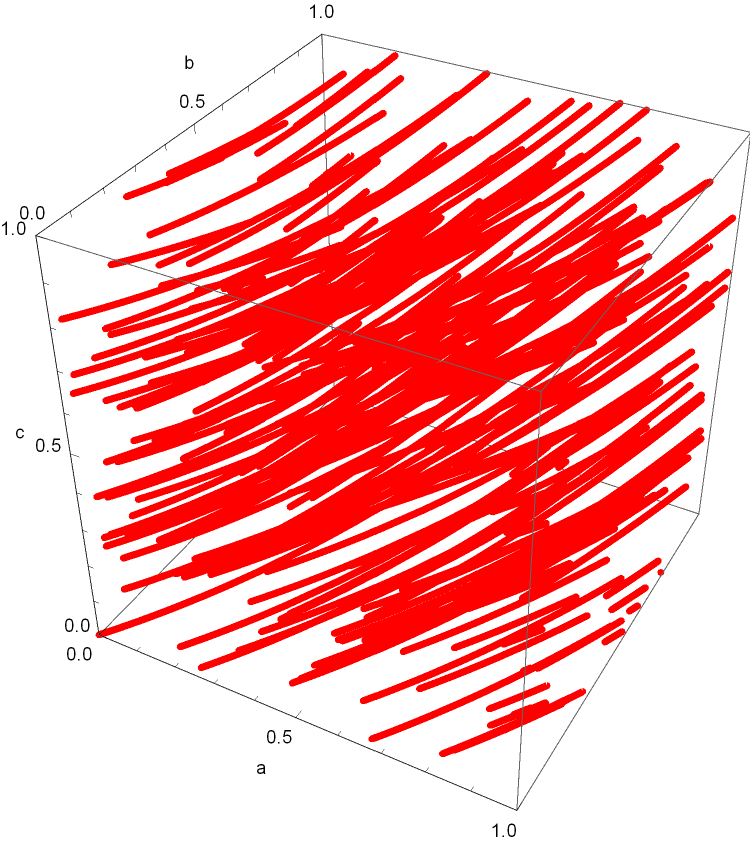}{The curve $\Gamma$ for $\tan \theta = \sqrt{2}$}{fig:geod6}{0.5}

As it is shown below, Theorem \ref{tgdense} is implied by the following theorem.

\begin{theorem} \label{conj} The sequence $\{(\{2rn\}, \{ rn^2\})\in\tt^2 \ | \ n\in\nn\}$ is dense in $\tt^2$ for every $r\in\rr\setminus\qq$. 
\end{theorem}
Theorem \ref{conj} follows from a more general  result due to H.Furstenberg, see \cite[lemma 2.1]{furst}, which yields unique ergodicity of the torus map (\ref{tab}). Below we present a proof of Theorem \ref{conj} for completeness of presentation.

\begin{remark} It is known that for every real polynomial $P(n)=\alpha_0 n^m+\alpha_1 n^{m-1}+\dots
+\alpha_m$ with $\alpha_0\notin\qq$ the values $P(n)$ are equidistributed (thus, dense) 
on the segment $[0,1]$ (Furstenberg's theorem, see \cite[exercise 4.2.7]{kathas}). This theorem also follows from the above-mentioned Furstenberg's result on unique ergodicity of map (\ref{tab}).\end{remark}

\begin{proof} {\bf of Theorem \ref{tgdense} modulo Theorem \ref{conj}.} We prove the  statement of 
Theorem \ref{tgdense} for the half-curve $\Gamma_+$; the proof for $\Gamma_-$ is analogous.

Let us do the above calculation with 
$$T=\frac q{\cos\theta}, \ \ q\in\zz_{\geq0}:$$
$$T\cos\theta=q\in\zz_{\geq0}, \ \ T\sin\theta=q\tan\theta,$$
$$c'(t+T)=\{ c(t+T)-[a(t+T)]b(t+T)\}$$
$$=\left\{\frac{(t+T)^2}2\sin\theta\cos\theta-[(t+T)\cos\theta](t+T)\sin\theta\right\}$$
$$=\left\{\frac{t^2}2\sin\theta\cos\theta+tT\sin\theta\cos\theta+\frac{T^2}2\sin\theta\cos\theta- 
[(t+T)\cos\theta](t+T)\sin\theta\right\}$$
$$=\left\{\frac{t^2}2\sin\theta\cos\theta+tq\sin\theta+\frac{q^2\tan\theta}2-([t\cos\theta]+q)(t\sin\theta+q\tan\theta)\right\}$$
\begin{equation}=c'_q(t):= \left\{\frac{t^2}2\sin\theta\cos\theta-[t\cos\theta]t\sin\theta-[t\cos\theta]q\tan\theta-
q^2\frac{\tan\theta}2\right\}.\label{cqt}\end{equation}

Set 
$$\wh\Gamma_q:=\{\Gamma(t+\frac q{\cos\theta}) \ | \ t\in [0,\frac1{\cos\theta})\} \ \ \text{ for every } \ q\in\zz_{\geq0}.$$ 
By definition, 
\begin{equation}\Gamma_+\cup\{ g'(0)\}=\cup_{q\in\zz_{\geq0}}\wh\Gamma_{q}.\label{gu}\end{equation} 
 Set 
$$r=\frac{\tan\theta}2.$$
Each curve $\wh\Gamma_q$ admits the coordinate representation  
\begin{equation}\wh\Gamma_{q}(t)=(a'(t), \{ b'(t)+2qr\}, \{ c'(t)-rq^2\}),\label{gqc}\end{equation}
since for every $t\in[0,\frac1{\cos\theta})$ one has 
$$c'_{q}(t)=\left\{\frac{t^2}2\sin\theta\cos\theta-[t\cos\theta]t\sin\theta-
q^2\frac{\tan\theta}2\right\}=\{ c'(t)-rq^2\}:$$
$[t\cos\theta]=0$, whenever $t\in[0,\frac1{\cos\theta})$. 
For every $t$ the sequence $(\{ b'(t)+2qr\}, \{ c'(t)-rq^2\})$  is dense in $[0,1)\times[0,1)$.  
To prove this, it suffices to show that the sequence $(\{2qr\},\{ -rq^2\})$ is dense in $[0,1)\times[0,1)$. Or equivalently, density of the projection to $\tt^2=\rr^2\slash\zz^2$ 
of the sequence $(2qr,-rq^2)$. Indeed, the projection to $\tt^2$  of the sequence $(2qr,rq^2)$ 
is dense, by Theorem \ref{conj}. The sequence $(2qr,-rq^2)$ is obtained from the latter 
sequence with dense projection by the symmetry $(x,y)\mapsto(x,-y)$, which is the lifting 
to $\rr^2$ of torus automorphism given by the same formula. Every torus  automorphism 
sends any dense subset to a dense subset. The latter symmetry sends the projection of the 
sequence $(2qr,rq^2)$ to the projection of the sequence  $(2qr,-rq^2)$. Therefore, the 
latter projection is dense. Hence, the sequence $(\{ b'(t)+2qr\}, \{ c'(t)-rq^2\})$  is dense in $[0,1)\times[0,1)$ for every $t\in[0,\frac1{\cos\theta})$. This together with (\ref{gqc}) and (\ref{gu}) implies density of the curve $\Gamma_+$ in $M$. 
\end{proof}

For the proof of Theorem \ref{conj} let us introduce the torus map 
\begin{equation} T:\tt^2\to\tt^2, \ T(x,y)=(x+\alpha, y+x+\beta).\label{tab}\end{equation}
\begin{proposition} Set $\alpha=2r$, $\beta=r$. Then 
\begin{equation} (\{2rn\}, \{ rn^2\})=T^n(0,0).\label{2rnn}\end{equation}
\end{proposition}
\begin{proof} Induction in $n$. 

Induction base: for $n=1$ one has $T(0,0)=(2r,r)(\mod\zz^2)$. 

Induction step. Let $T^n(0,0)=(2rn,rn^2)(\mod\zz^2)$. Then modulo $\zz^2$, 
$$T^{n+1}(0,0)=(2r(n+1), rn^2+2rn+r)=(2r(n+1),r(n+1)^2).$$
The induction step is done. The proposition is proved.
\end{proof}

\begin{theorem} \label{tmin} 
For every $\alpha\in\rr\setminus\qq$ and every $\beta\in\rr$ the map $T$ given 
by (\ref{tab}) is minimal: each its forward orbit is dense. 
\end{theorem}
\begin{proof} Suppose the contrary: there exists a point $(x_0,y_0)$ with non-dense forward 
orbit. Let $M$ denote the set of limit points of its orbit.  (The orbit is non-periodic, as is the rotation $x\mapsto x+\alpha$.) The set $M$ is a non-empty 
closed subset in $\tt^2$ with 
a non-empty open complement 
$$V:=\tt^2\setminus M.$$

\begin{remark} The set $M$ is $T$- and $T^{-1}$-invariant, hence so is $V$.
\end{remark}
\begin{proposition} \label{profib} The set $V$ contains no fiber $z\times S^1$. 
\end{proposition} 
\begin{proof} Suppose the contrary: $V$ contains such a fiber. Then there exists an interval 
 neighborhood 
$U=U(z)\subset S^1$ such that $U\times S^1\subset V$. But the successive images $T^m(U\times S^1)$ cover all of $\tt^2$: the images of the interval $U$ by translations 
$x\mapsto x+m\alpha$ cover all of $S^1$, since $\alpha$ is irrational. Therefore, $V=\tt^2$. 
The contradiction thus obtained proves the proposition.
\end{proof}

Fix $a,b, c, d\in[0,1)$, $a<b$, $c<d$, such that 
$$\Pi:=[a,b]\times[c,d]\subset V.$$
 Set 
$$h:=d-c.$$
\begin{proposition} \label{proarc} For every $k\in\zz_{\geq0}$ there exists a fiber $S^1_k=z_k\times S^1$ 
such that $S^1_k\cap V$ contains an arc of length greater than $2^{k}h$. In the case, when 
$2^kh>1$, this means that the whole fiber $S^1_k$ lies in $V$. 
\end{proposition} 
\begin{proof} The proof is based on area-preserving property of the map $T$ and 
the fact that for every $N\in\nn$ the iterate $T^N$ lifted to 
$\rr^2$ transforms horizontal lines to lines with the slope (the tangent of angle with the horizontal axis) equal to $N$. Thus, as $N\to\infty$, the images of horizontal lines tend to vertical lines. 
Therefore, the images of a rectangle become  very long strips spiralling in nearly vertical direction. 

Induction in $k$. 

Induction base for $k=0$. Each fiber 
$z\times S^1$, $z\in[a,b]$, intersects $V$ by an arc strictly containing the arc $z\times[c,d]$ of 
length $h$, and thus, having a bigger length. 

Induction step. Let there exist a $z_k\in S^1$ such that the intersection $(z_k\times S^1)\cap V$ 
contains a segment $z_k\times[c_k,d_k]$ of length $c_k-d_k>2^kh$. Let us show that $V$ contains a vertical circle arc of twice bigger length.  To do this, fix an $\var>0$ such that 
$$\Pi_k:=[z_k-3\var,z_k+3\var]\times[c_k,d_k]\subset V.$$
Fix a $\delta\in(0,\var)$ such that $2\delta<h$. For every $N\in\nn$ there exists a point 
$$(x_N,y_N)\in \Pi_{k,\delta}:=[z_k-\var,z_k+\var]\times[c_k,c_k+\delta]\subset\Pi_k$$
such that 
$$p_n:=T^n(x_N,y_N)\in\Pi_{k,\delta} \text{ for some } n>N,$$
 by area-preserving property and 
the Poincar\'e Recurrence Theorem \cite[theorem 4.1.19]{kathas}. 

\smallskip

{\bf Claim 1.} {\it Let us  choose 
$N>\frac{d_k-c_k+2\delta}\var$. Let $(x_N, y_N)$ and $n$ be as above.  Then the $T^n$-image of the lower horizontal side 
$L:=[z_k-3\var,z_k+3\var]\times c_k$  of the rectangle $\Pi_k$ intersects its  upper 
 side at some point $q_k$. See Fig.12.}
\begin{figure}[ht]
  \begin{center}
   \epsfig{file=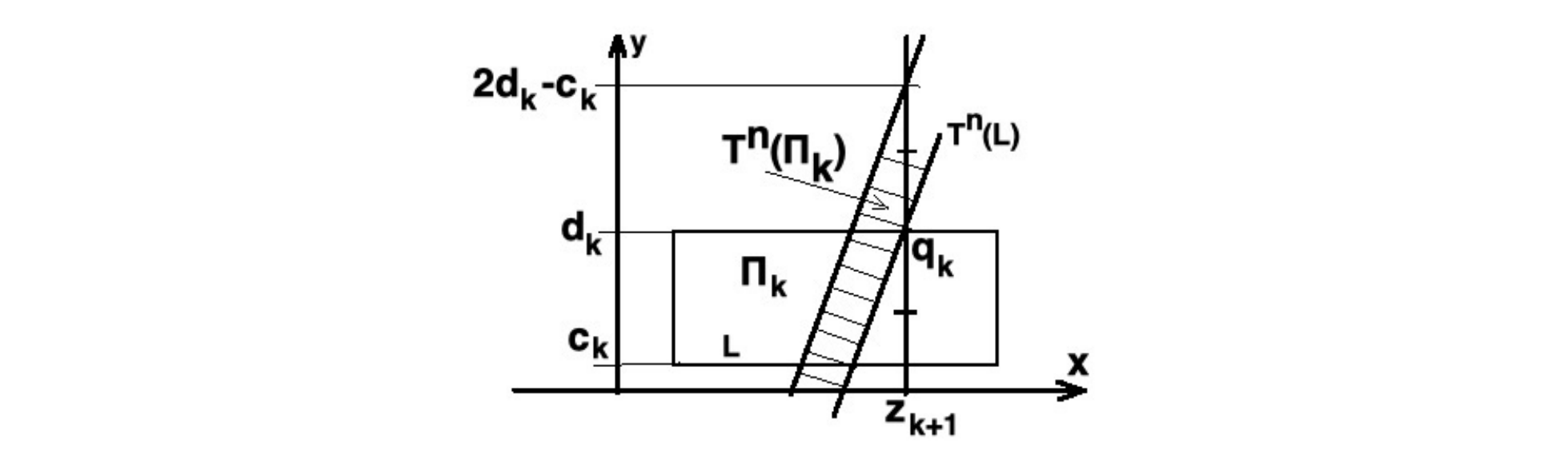, width=33em}
    \caption{The rectangle $\Pi_k$, its image $T^n(\Pi_k)$ and the intersection point $q_k$.}
    \label{fig:01}
  \end{center}
\end{figure}

\begin{proof} The point $(x_N,c_k)$ lies in the lower side $L$, $|x_N-z_k|\leq\var$, and 
$0\leq y_N-c_k\leq\delta$. The $y$-coordinate of its image,  
$y_{n}:=y(T^n(x_N,c_k))$ also differs from $c_k$ by a quantity no greater than $\delta$, since 
it is no greater than that  $y_{n}':= y(T^n(x_N, y_N))\in[c_k,c_k+\delta]$, $y'_n-y_n=y_N-c_k
\in[0,\delta]$: the map  $T$ preserves the lengths 
of arcs of vertical fibers. The $x$-coordinate of the same image $T^n(x_N,c_k)$ lies in the segment $[z_k-\var,z_k+\var]$. This together with the inequality on $N$ and the fact that the $T^n$-image of a horizontal 
segment has slope $n>N$ 
implies that $T^n(L)$ crosses the upper side of the rectangle $\Pi_k$. 
\end{proof}
  
Let $q_k=(z_{k+1},d_k)=T^n(p_k)$ be a point of the above crossing, $p_k=(s_k,c_k)\in L$. Then $s_k\times[c_k,d_k]\subset \Pi_k\subset V$, 
by assumption. Therefore, $z_{k+1}\times[d_k, d_k+(d_k-c_k)]=T^n(s_k\times[c_k,d_k])\subset V$, by invariance of the set $V$. Finally, the vertical circle 
arc $z_{k+1}\times[c_k, 2d_k-c_k]$ of length 
$2(d_k-c_k)>2^{k+1}h$ lies in $V$. The induction step is done. Proposition \ref{proarc} is proved.
\end{proof} 

Proposition \ref{proarc} applied to $k$ large enough implies that $V$ contains a vertical fiber 
$z\times S^1$. This contradicts Proposition \ref{profib}. Theorem \ref{tmin} is proved.
\end{proof} 

Theorem \ref{tmin} implies Theorem \ref{conj}, and hence, Theorem \ref{tgdense}.

\section[Projections of   geodesics-spirals to Heisenberg nil-manifold]{Projections of   geodesics-spirals \\to Heisenberg nil-manifold }\label{sec:spirals}

For every $\theta\in\rr\slash(2\pi\zz)$, $h \neq 0$ consider a geodesic-spiral 
 \eqref{x2}--\eqref{z2}  in $G$ 
and its projection to $M$: 
\begin{align*}
&a(t) = (\sin (\theta + h t) - \sin \theta) / h, \\
&b(t) = (\cos \theta - \cos (\theta + h t)) / h, \\
&c(t) = (h t - \sin h t) / (2 h^2)  + a(t) b(t)/2,
\end{align*}
$$g(t)=(a(t), b(t), c(t)), \ \ 
t\in\rr.$$
$$g'(t):= \pi\circ g(t)=(a'(t), b'(t), c'(t)),$$
$$\Gamma:=\{ g'(t) \ | \ t\in\rr\}\subset M.$$

The projection $G\to \rr\times\rr$,  $(a, b, c)\mapsto (a,b)$ passes to the quotient and 
induces the projection 
$$p:M\to\tt^2=S^1\times S^1, \ \ S^1=\rr\slash\zz.$$
 
\begin{theorem} \label{tspirals} 1) The projection $p$ sends each geodesic-spiral $\Gamma$, see 
\eqref{x2}--\eqref{z2},  to a contractible closed 
curve  $\gamma\subset\tt^2$ that may have self-intersections. 

2) The geodesic-spiral $\Gamma$ is 

- either closed, which holds if and only if $h^2\in\pi\mathbb Q \setminus \{0\}$;

- or dense in the preimage $p^{-1}(\gamma)\subset M$, if 
 $h^2\notin\pi\mathbb Q$. 
 \end{theorem}

See Figs. \ref{fig:geod9n}--\ref{fig:geod8}. 

\twofiglabelsize{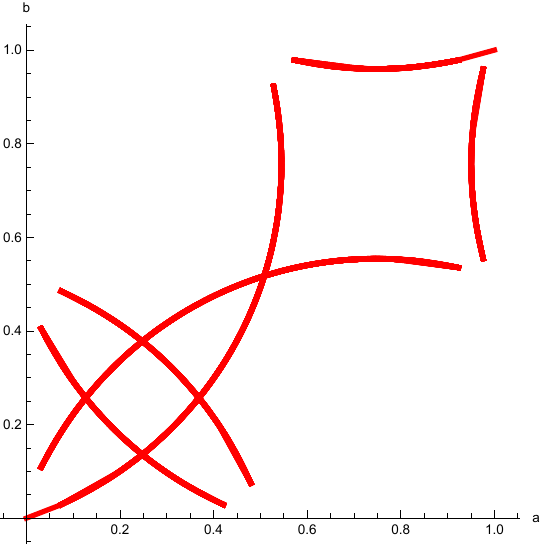}{The curve $\gamma$ for $h = \sqrt{\pi/2}$}{fig:geod9n}{0.5}
{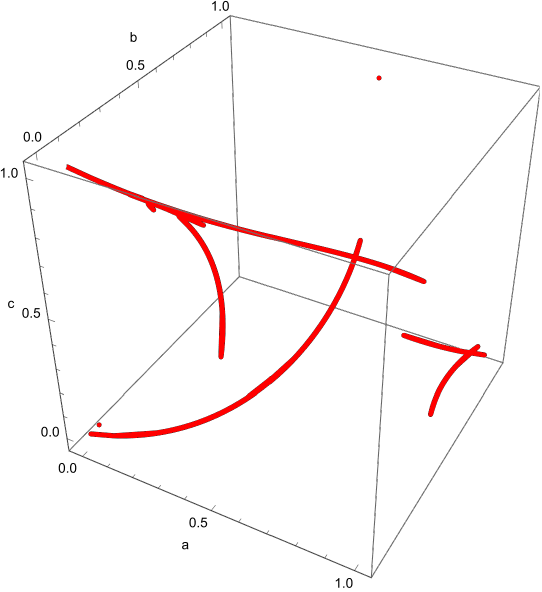}{The curve $\Gamma$ for $h = \sqrt{\pi/2}$}{fig:geod10n}{0.5}

\twofiglabelsize{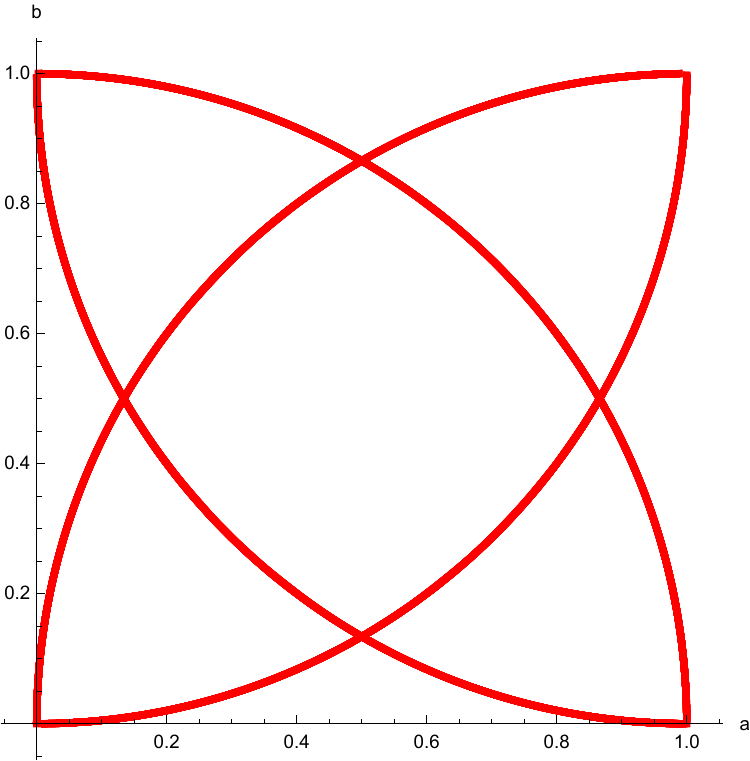}{The curve $\gamma$ for $h = 1$}{fig:geod7}{0.5}
{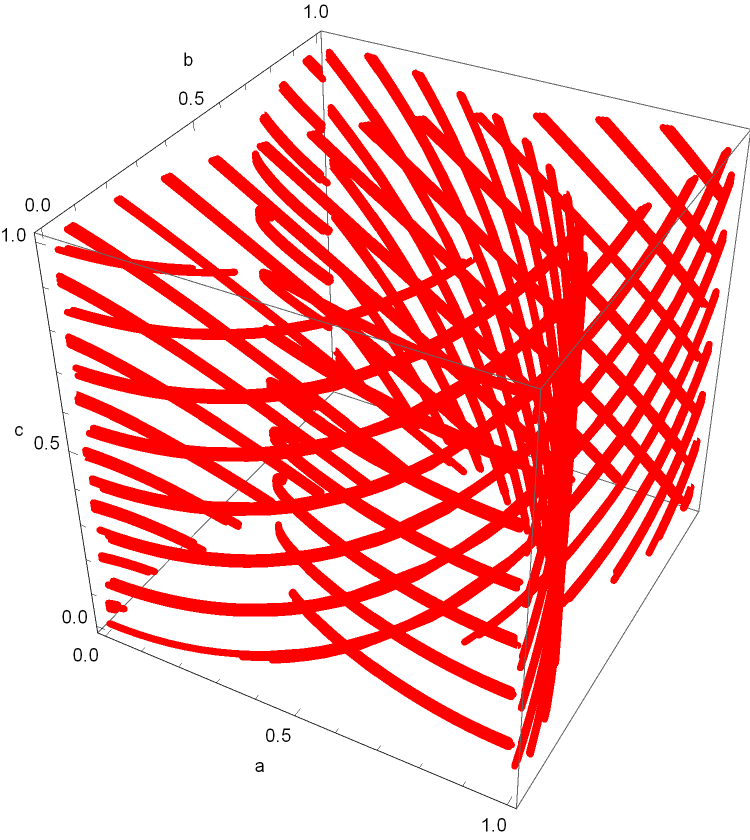}{The curve $\Gamma$ for $h = 1$}{fig:geod8}{0.5}

 \begin{proof} The projection $p(\Gamma)\subset\tt^2$ of the spiral geodesic $\Gamma$ given by  \eqref{x2}--\eqref{z2} is $\frac{2\pi}h$-periodic in $t$, as are the right-hand sides in \eqref{x2} 
 and \eqref{y2}. Therefore, it is a closed curve. Statement 1) is proved.
 
 The coordinate $c'=\{ c-[a]b\}$ of a point of the geodesic $\Gamma$ is equal to 
\begin{equation}c'(t)=\left\{ \frac t{2h}-\frac{\sin ht}{2h^2}+\frac{x(t)y(t)}2-[x(t)]y(t)\right\},\label{c'sp}
\end{equation}
 where $x(t)$, $y(t)$ are given by \eqref{x2} and \eqref{y2} respectively. All the terms in 
the right-hand side in  (\ref{c'sp}) except for the first one are $\frac{2\pi}h$-periodic functions 
in $t$. Adding $\frac{2\pi}h$ to $t$ results in adding $\frac{\pi}{h^2}$ to the first term $\frac t{2h}$. 
Therefore, closeness of the geodesic $\Gamma$ is equivalent to commensurability of the 
numbers $h^2$ and $\pi$. If they are incommensurable, then the sequence of the 
numbers $\{\frac{\pi n}{h^2}\}$ with $n\in\nn$ is dense in $[0,1]$. This together with the above discussion implies that $\Gamma$ is dense in $p^{-1}(\gamma)$. The theorem is proved.
\end{proof}

\section[Dynamics of the normal Hamiltonian flow on $T^*M$]{Dynamics of the normal Hamiltonian flow \\on $T^*M$}\label{sec:dyn}
Let $T^*M$ be the cotangent bundle of the Heisenberg nil-manifold $M$. Introduce linear on fibers of $T^*M$ Hamiltonians $h_i(\lam) = \langle\lam, X_i'\rangle$, $i = 1, 2, 3$, where $X_3' = [X_1', X_2'] =\frac{\partial}{\partial z'}$.

Let $H(\lam) = (h_1^2(\lam)+h_2^2(\lam))/2$ be the normal Hamiltonian of Pontryagin maximum principle \cite{ABB_book, intro} for the sub-Riemannian problem \eqref{SR'1}--\eqref{SR'4}, and let $\vH$ be the corresponding Hamiltonian vector field on~$T^*M$.
Sub-Riemannian geodesics on $M$ are projections of trajectories of the normal Hamiltonian system
\begin{align}
&\dot \lam = \vH(\lam), \qquad \lam \in T^*M,
\intertext{in coordinates}
&\dot h_1 = - h_2 h_3, \quad \dot h_2 = h_1 h_3, \quad \dot h_3 = 0, \quad \dot g' = h_1 X_1' + h_2 X_2', \label{dh1}
\end{align} 
this follows from the classical coordinate expression of the Hamiltonian system for sub-Riemannian geodesics on the Heisenberg group \cite{ABB_book, intro}.

Each level surface
$$
S_p = \{\lam \in T^*M \mid H(\lam) = 1/2, \ h_3(\lam) = p\}, \qquad p \in \rr,
$$
is invariant for the field $\vH$. 
Denote on this level surface $h_1 = \cos \theta$, $h_2 = \sin \theta$.
Denote also the restriction $V_p = \restr{\vH}{S_p}$. 
The ODE 
\be{dlam}
\dot \lam = V_p(\lam), \qquad \lam \in S_p,
\ee
 reads in coordinates as
\begin{equation}\begin{cases}
\dot \theta = p, \\
\dot a' = \cos \theta, \\
\dot b' = \sin \theta, \\ 
\dot c' = a' \sin \theta,\end{cases}\label{dlamn}
\end{equation}
this follows immediately from  ODEs \eq{dh1} via the transformation formulas~\eq{abcxyz}.

Set 
$$S^1_\nu:=\rr_\nu\slash\zz \ \ \ \text{ for } \nu=a, b, c; \ \ 
S^1_\theta=\rr_{\theta}\slash2\pi\zz.$$
Let us introduce yet another flow on $S_p=S^1_\theta\times M$. 

\begin{definition} The {\it $p$-standard flow} $F^t_p$ 
on $S^1_\theta\times M$ is given in the coordinates by the equation
\begin{equation}\begin{cases} \dot\theta=p,\\
\dot a'=0,\\ \dot b'=0, \\ \dot c'=\frac{1}{2p}.\end{cases}\label{stpflow}\end{equation}
\end{definition}
\begin{remark} The projection of the vector field given by (\ref{stpflow}) to the 
2-torus $\tt^2_{\theta,c'}$ is the linear vector field $\dot\theta=p$, $\dot c'=\frac{1}{2p}$. 
Its  flow map in time $\frac{2\pi}p$ fixes each circle $\{\theta\}\times S^1_{c'}$ and 
acts on it by translation (rotation) 
$$c'\mapsto c'+\rho, \ \ \ \ \   \rho:=\frac{\pi}{p^2}.$$
The number $\rho$ is its rotation number, see the definition of rotation number in \cite[p. 104]{arn}. 
\end{remark}
\begin{theorem} \label{tnflow} Flow (\ref{dlam}) is conjugated to the $p$-standard flow by 
a diffeomorphism of $S_p=S^1_\theta\times M$  preserving the 
$\theta$-coordinate and isotopic to the identity in the class of diffeomorphisms preserving 
the $\theta$-coordinate.
\end{theorem}
\begin{remark} \label{rkpul} Let $\pi:G\to M$ denote the quotient projection. Consider 
the pullbacks to $S^1_\theta\times G$ of vector fields (\ref{dlamn}), (\ref{stpflow}) 
under the induced projection $S^1\times G\to S^1\times M$. 
The pullbacks are written by the same formulas, as  (\ref{dlamn}), (\ref{stpflow}), but in the coordinates $(\theta,a,b,c)$.
\end{remark}

For the proof of Theorem \ref{tnflow}  we first solve equations (\ref{dlamn}) and show that their solutions are all $\frac{2\pi}p$-periodic, except for the function $c(t)$, which is 
equal to $\frac{t}{2p}$ plus a $\frac{2\pi}p$-periodic function. Using formulas for solutions,  we construct an explicit analytic family of diffeomorphisms  $F_\nu:S^1_\theta\times G\to 
S^1_\theta\times G$ 
preserving the $\theta$-coordinate, depending on the parameter $\nu\in[0,1]$,  $F_0=\Id,$ 
such that $F_1$ transforms the lifted vector field (\ref{dlamn}) to the lifted field (\ref{stpflow}), see the above remark. We show that each $F_\nu$ is $\Gamma$-equivariant and thus, the family $F_\nu$ induces a family of diffeomorphisms 
$S^1\times M\to S^1\times M$ with $F_1$ sending (\ref{dlamn}) to (\ref{stpflow}). 
This will prove Theorem \ref{tnflow}. 

Set 
$$\Sigma:=\{ H=\frac12\}\subset T^*M.$$

\begin{theorem} \label{tint} 1) The Hamiltonian flow on 
$T^*M^o:=T^*M\setminus\{ H=0\}$ with Hamiltonian function $H$ is integrable on the 
invariant domain $T^*M^{o,h_3}:=T^*M^o\setminus\{ h_3=0\}$: it has an additional integral $I(\lambda,a,b,c)$ analytic on $T^*M^{o,h_3}$ 
  that is in involution with 
the integrals $H$ and $h_3$ for the canonical symplectic structure on $T^*M$. 
The latter integral $I$ can be chosen to be any of the following functions:
\begin{equation} \cos(2\pi(a-\frac{\sin\theta}{h_3})), \ \sin(2\pi(a-\frac{\sin\theta}{h_3})), \  \cos(2\pi(b+\frac{\cos\theta}{h_3})), \  \sin(2\pi(b+\frac{\cos\theta}{h_3})).\label{ints}\end{equation}

2) For every integral $I$ from (\ref{ints}) and  $\nu\in(-1,1)$, $p\neq0$, the manifold    
$$S_{p,\nu}:=\{ H=\frac12\}\cap\{ h_3=p\}\cap\{ I=\nu\}=S_p\cap\{ I=\nu\}$$ 
is a transversal intersection, a disjoint union of two invariant 3-tori.

3) For $\nu=\pm1$ (i.e., when $\nu$ is an extremum of the integral $I$) the latter tori coincide and the above 
intersection is one 3-torus.

4) The restriction of the Hamiltonian flow to any of the two latter tori is conjugated to a constant vector field 
on the standard 3-torus $\rr^3\slash2\pi\zz^3$ with closed orbits for $p^2\in\pi\mathbb Q \setminus \{0\}$ 
and each orbit dense in a 2-torus for $p^2\notin\pi\mathbb Q$.

5) The  flow  restricted to the hypersurface $\Sigma=\{ H=\frac12\}$ has no non-trivial analytic integral: each analytic 
integral is a function of $h_3$. 
\end{theorem}

Theorems \ref{tnflow}, \ref{tint} are proved in Subsections 7.1 and 7.2 respectively.

\subsection{Conjugacy with $p$-standard flow. Proof of Theorem \ref{tnflow}}

\begin{proposition} Each solution of the differential equation defined by the lifted vector field (\ref{dlamn}) (thus, written in the coordinates $(\theta,a,b,c)$) with initial condition 
$(\theta_0,a_0, b_0, c_0)$ at $t=0$ takes the form 
\begin{equation} \begin{cases}\theta(t)=\theta_0+pt\\
a(t)=\frac1p(\sin(\theta_0+pt)-\sin\theta_0)+a_0\\
b(t)=\frac1p(\cos\theta_0-\cos(\theta_0+pt))+b_0\\
c(t)=\frac t{2p}-\frac1{4p^2}(\sin2(\theta_0+pt)+\sin2\theta_0)+
\frac1{2p^2}(\sin(2\theta_0+pt)-\sin pt)\\ -\frac{a_0}p(\cos(\theta_0+pt)
-\cos\theta_0)+c_0.\end{cases}\label{eqsol}\end{equation}
\end{proposition}
\begin{proof} The three first equations in (\ref{dlamn}) are solved by direct integration. 
The fourth equation is solved by taking the primitive:
$$c(t)=c_0+\int_0^ta(t)\sin\theta(t)dt$$
\begin{equation}=c_0+\int_0^t(\frac1p(\sin(\theta_0+pt)-\sin\theta_0)+a_0)\sin(\theta_0+pt)dt.\label{c't}\end{equation}
The latter subintegral expression is equal to 
$$\frac1{2p}(1-\cos2(\theta_0+pt)+\cos(2\theta_0+pt)-\cos pt)+a_0\sin(\theta_0+pt).$$
Therefore, the integral is equal to 
$$\frac t{2p}-\frac1{4p^2}(\sin2(\theta_0+pt)-\sin2\theta_0+2\sin2\theta_0-2\sin(2\theta_0+pt)+2\sin pt)$$
$$+
\frac{a_0}p(\cos\theta_0-\cos(\theta_0+pt)).$$
This together with (\ref{c't}) yields (\ref{eqsol}).
\end{proof}

\begin{proposition} The phase curves of the lifted vector field (\ref{dlamn}) are 
graphs of vector functions $(a(\theta),b(\theta),c(\theta))$, where 
\begin{equation} \begin{cases}a(\theta)=\frac{\sin\theta}p+a_0\\
b(\theta)=\frac1p(1-\cos\theta)+b_0\\
c(\theta)=\frac\theta{2p^2}-\frac1{4p^2}\sin2\theta-\frac{a_0}p(\cos\theta
-1)+c_0,\end{cases}\label{eqsol2}\end{equation}
$$a_0=a(0), b_0=b(0), c_0=c(0).$$
\end{proposition}
\begin{proof} Each phase curve intersects the fiber $\{\theta=0\}$. Hence, it is 
the graph of a solution (\ref{eqsol}) with $\theta_0=0$.
Substituting $\theta_0=0$ to (\ref{eqsol}) yields 
$$ \begin{cases}\theta(t)=pt\\
a(t)=\frac{\sin\theta}p+a_0\\
b(t)=\frac1p(1-\cos\theta)+b_0\\
c(t)=\frac t{2p}-\frac1{4p^2}\sin2\theta-\frac{a_0}p(\cos\theta
-1)+c_0,\end{cases}$$
which implies (\ref{eqsol2}).
\end{proof}

Consider the following family of diffeomorphisms  $F_\nu:S^1_\theta\times G\to 
S^1_\theta\times G$:
$$ F_\nu(\theta,a,b,c)=(\theta,\ a-\frac{\nu}p\sin\theta, \ 
b-\frac{\nu}p(1-\cos\theta), \ \wt c_\nu),$$
\begin{equation}\wt c_\nu:=
c+\nu\left(\frac1{4p^2}\sin2\theta+\frac1p(a-\frac{\sin\theta}p)(\cos\theta
-1)\right), \ \nu\in[0,1].\label{fnu}\end{equation}
The action of the group $\Gamma$ on $G$ by left multiplication lifts to its action on 
$S^1_\theta\times G$: 
$$\gamma(\theta,g):=(\theta,\gamma g) \ \text{ for every } \ \gamma\in\Gamma, \ g\in G.$$
\begin{proposition} \label{pequiv}
 1) For every $\nu\in\rr$ the map $F_\nu$ is a diffeomorphism  equivariant under left multiplications by elements of the group 
$\Gamma$: 
\begin{equation}F_\nu(\theta,\gamma g)=\gamma F_\nu(\theta, g) \text{ for every } \gamma\in\Gamma, \ g\in G.
\label{equivg}\end{equation}
2) One has $F_0=id$, and $F_1$ transforms the lifting to $S^1\times G$ of flow 
(\ref{dlamn}) to the lifting of the $p$-standard flow (\ref{stpflow}). 
\end{proposition}
\begin{proof} Statement 2) follows from (\ref{fnu}) and (\ref{eqsol2}). Let us prove 
Statement 1). The group $\Gamma$ has two generators:
$$\mca:=\left(\begin{matrix} 1 & 1 & 0\\ 0 & 1 & 0\\ 0 & 0 & 1\end{matrix}\right), 
\ \ \mcb:=\left(\begin{matrix} 1 & 0 & 0\\ 0 & 1 & 1\\ 0 & 0 & 1\end{matrix}\right).$$
Let us represent each $g\in G$ by its coordinates $(a,b,c)$. The multiplication by 
$\mca$ from the left acting on $G$ lifts to the action on $S^1_\theta\times G$ 
 by the formula $(\theta,a,b,c)\mapsto(\theta,a+1,b,c+b)$. Therefore,
$$\mca F_\nu(\theta,a,b,c)=(\theta, a+1-\frac{\nu}p\sin\theta, b-\frac{\nu}p(1-\cos\theta), \wt c_{\nu,1}),$$
\begin{equation}\wt c_{\nu,1}:=c+\nu\left(\frac1{4p^2}\sin2\theta+\frac1p(a-\frac{\sin\theta}p)(\cos\theta
-1)\right)+b+\frac{\nu}p(\cos\theta-1),\label{cnu1}\end{equation}
$$F_\nu\circ\mca(\theta,a,b,c)=F_\nu(\theta,a+1,b,c+b)=(\theta, a+1-\frac{\nu}p\sin\theta, b-\frac{\nu}p(1-\cos\theta), \wt c_{\nu,2}),$$
$$\wt c_{\nu,2}:=c+\nu\left(\frac1{4p^2}\sin2\theta+\frac1p(a+1-\frac{\sin\theta}p)(\cos\theta
-1)\right)+b.$$
This together with (\ref{cnu1}) implies that $c_{\nu,1}=c_{\nu,2}$ and proves (\ref{equivg}) 
for $\gamma=\mca$. Statement (\ref{equivg}) for $\gamma=\mcb$ follows from 
(\ref{fnu}) and the relation $\mcb(\theta,a,b,c)=(\theta,a,b+1,c)$. Proposition \ref{pequiv} 
is proved.
\end{proof}

The quotient of the diffeomorphism $F_1$ under the projection 
$\pi:S^1_\theta\times G\to S^1_\theta\times M$ is a well-defined diffeomorphism of the 
manifold $S_p=S^1_\theta\times M$ preserving the coordinate $\theta$ and 
isotopic to the identity in the class of diffeomorphisms preserving the coordinate $\theta$. 
It transforms flow (\ref{dlamn}) to (\ref{stpflow}) by construction. This proves 
Theorem \ref{tnflow}.

\subsection{Integrability. Proof of Theorem \ref{tint}}
Let us prove Statement 1) of Theorem \ref{tint}. 
The function $f=a-\frac{\sin\theta}{h_3}$ is defined on the cotangent bundle $T^*G$ to the group $G$ 
with the hypersurface $\{ h_3=0\}$ deleted. It 
  is a first integral of flow (\ref{dlamn}) 
lifted to $T^*G$. This follows immediately from the first and second equation in (\ref{dlamn}), since 
$h_3$ is a first integral. Similarly the function $g=b+\frac{\cos\theta}{h_3}$ is an integral. Each one of the two latter functions is automatically in involution with the Hamiltonian $H$, being an integral. Let us show that $f$ 
is in involution with $h_3$, i.e., 
\begin{equation}\omega(\vec f, \vec h_3)=0, \ \omega \ \text{ is the standard symplectic form on } \ 
T^*G.\label{ininv}\end{equation}
Here by $\vec \psi$ we denote the skew gradient of a function $\psi$, which means by definition that $\omega(\vec\psi, v)=(d\psi)(v)$ 
for every $v\in TM$. Indeed, consider the coordinates $(x,y,z)$ on $G$ and the associated coordinates $(x,y,z,\la_1,\la_2,\la_3)$ on $T^*G$: $\la_1$, $\la_2$, $\la_3$ 
is the basis in $T^*_{(x,y,z)}G$  dual to the basis $\frac{\partial}{\partial x}$, $\frac{\partial}{\partial y}$,  $\frac{\partial}{\partial z}$ in $T_{(x,y,z)}G$. The standard symplectic form on $T^*G$ is 
$$\omega=dx\wedge d\la_1+ dy\wedge d\la_2+dz\wedge d\la_3.$$
One has 
\begin{equation}\vec h_3=\vec\la_3=(0, 0, 1, 0, 0, 0).\label{nabh}\end{equation}
To calculate $\vec f$, recall that 
$$h_1=\langle\la, X_1\rangle=\la_1-\frac y2\la_3, \ \  h_2=\langle\la, X_2\rangle=\la_2+\frac x2\la_3, \ \ 
\theta=\arctan\frac{h_2}{h_1},$$
$$f=x-\frac{\sin\theta}{\la_3}.$$
The sixth, that is, $\la_3$-component of the skew gradient $\vec f$ is zero, since $f$ is $z$-independent. Therefore, $\omega$ vanishes on the pair of vectors $\vec h_3$, $\vec f$, 
by (\ref{nabh}). This proves (\ref{ininv}). Analogously, the function $g$ is in involution with $h_3$. 
This implies that each one of integrals in (\ref{ints}) is in involution with $h_3$ on $T^*M^o$ and proves Statement 1).

Let us prove Statement 2).  The submanifold $S_p=\{H=\frac12 , \ h_3=p\}=S^1\times M\subset T^*M$ 
lifts to $T^*G$ as a submanifold $\wt S_p=S^1\times G$ covering $S_p$ via the canonical projection $\wt S_p\to S_p$ induced by the quotient projection $G\to M$. 
The function 
$f=a-\frac{\sin\theta}{h_3}=a-\frac{\sin\theta}p$ is well-defined on $\wt S_p$ for $p\neq0$. 
The surface $\wt S_p$  has natural coordinates $(\theta, a, b, c)$. The function $f|_{\wt S_p}$ 
 has nowhere vanishing differential, since it has unit partial derivative in $a$.  Therefore,  the function $I=\cos(2\pi(a-\frac{\sin\theta}{h_3}))$ from (\ref{ints}) restricted to
 $S_p\subset T^*M^o$ also has nowhere vanishing differential, except for the points where the $|\cos|$ 
 takes its maximal value 1. Hence, its level hypersurfaces $\{ I=\nu\}$ with $\nu\in(-1,1)$ are transversal to 
 $S_p$. Writing 
 $$\nu=\cos2\pi\alpha, \ \ \alpha\in(0,\frac12),$$
 we get that 
 \begin{equation}\{ I=\nu\}\cap S_p=\cup_{\pm}\{(\theta,a,b,c) \ | \ a=\frac{\sin\theta}p\pm\alpha(\mo\zz)\}.\label{a=a}\end{equation}
 The latter intersection is a union of two compact invariant 3-manifolds, each being the subset in the right-hand side with a given sign choice $\pm$. They are tori: this follows from  Arnold--Liouville Theorem  on integrable systems  \cite[chapter 10, section 49]{arnmet} and can be also proved directly.  This together with analogous statements for the other integrals from (\ref{ints}) (proved similarly) 
 proves Statement 2). 
 
 As $\nu=\pm1$, one has $\alpha\in\{0,\frac12\}$, and the two 3-tori in the union 
 in (\ref{a=a})  coincide, since in this case $\alpha\equiv-\alpha(\mo\zz)$. 
 This proves Statement 3). 
 
 Statement 4) follows from Theorem \ref{tspirals} and Arnold--Liouville Theorem. 
 
In the proof of  Statement 5) we use the following obvious corollaries of Theorem \ref{tgdense} on density of 
orbits in $S_0=\Sigma\cap\{ h_3=0\}$. To state them, let us introduce the following notations. For every $\theta_0\in\rr$, $p\in\rr$,  set 
$$M_{\theta_0,p}:=\Sigma\cap\{\theta=\theta_0\}\cap\{ h_3=p\} =\{\theta_0\}\times M\subset S_p=S^1\times M.$$
Recall that for every $\theta_0$ the fiber $M_{\theta_0,0}$ is an invariant manifold for the flow, since $\dot\theta=0$. 
\begin{proposition} 1) For every $\theta$ with $\tan\theta\notin\mathbb Q\cup\{\infty\}$ the 
 flow on $M_{\theta,0}$ is miminal: each orbit is dense. 
 
 2) For every $\theta$ as above and $\var>0$ there exists a $T=T_{\var,\theta}>0$ 
such that each 
 finite orbit of the flow on $M_{\theta,0}$ in times $t\in[0,T]$ is $\frac{\var}4$-dense in $M_{\theta,0}$: 
 this means that it intersects the $\frac\var4$-neighborhood of each point in $M_{\theta,0}$. 
\end{proposition}
\begin{proof} Theorem \ref{tgdense} immediately implies density, which in its turn together with compactness 
 implies Statement 2).
\end{proof}
\begin{proposition} \label{propp} For every $\var>0$ there exists a $\delta>0$ such that for every $p\in(0,\delta)$ 
each orbit of the restriction of the flow to $S_p$ is $\var$-dense in $S_p$. 
\end{proposition}
\begin{proof} Choose a finite collection of numbers $\theta_1,\dots,\theta_N\in S^1=\rr\slash2\pi\zz$ with 
 $\tan\theta_j\notin\mathbb Q\cup\{\infty\}$ that is $\frac\var4$-dense on $S^1$. There exists a $T>0$ 
 such that for every $j=1,\dots,N$ each finite orbit of the flow on $M_{\theta_j,0}$ is $\frac\var4$-dense in 
 $M_{\theta_j,0}$. The vector fields defining the flows on $S_0$ and $S_p$ with  $p\in(0,\delta)$ 
 (both identified with the same product $S^1_\theta\times M$) are $\delta$-close to each other. Therefore, 
 as $\delta$ is small enough (depending on $\var$ and $T$), for every $p\in(0,\delta)$ 
 the finite orbit of the flow on $S_p=S^1\times M$ in times 
 $t\in[0,T]$ starting at each point $(\theta_j,x)$, $x\in M$, has $\frac\var2$-dense projection to $M$ so that 
 $\theta(t)\in[\theta_j,\theta_j+\frac\var4]$ for every $t\in[0,T]$.  
 Along each full orbit in $S_p$ the coordinate $\theta$ takes all values, including the above $\theta_j$'s. 
 This together with the latter statement implies that it is $\var$-dense in $S_p$. The proposition is proved.
 \end{proof}
  
  Let us now prove Statement 5). Let the restriction  of the Hamiltonian flow to  $\Sigma$ have a non-constant analytic integral $I$. Let us show that $I$ is constant on each level hypersurface 
 $S_p=\Sigma\cap\{ h_3=p\}$ of the function $h_3|_\Sigma$. This will imply that $I$ is a function of $h_3$ 
 and prove Statement 5). 
 
 The function $I_p:=I|_{S_p}$ achieves its minimum on a compact invariant set  for the flow in $S_p=S^1\times M$, 
 which will be denoted by $O_{p}\subset S^1\times M$ (compactness).  
 For arbitrarily small  $\var>0$ there exists a $\delta>0$ such that for every $p\in(0,\delta)$ 
 the invariant set $O_{p}$ is $\var$-dense in $S^1\times M$, since it consists of full orbits and by 
 Proposition \ref{propp}. Thus, it accumulates to all of $S_0$, as $p\to0$.  One has $dI_p=0$ at all points in $O_p$, thus, at all points of an $\var$-dense subset accumulating to the whole hypersurface $S_0\subset\Sigma$. 
 This implies that $I\equiv const$ on $S_0$. 
 
 The manifold $\Sigma$ is identified with the product $S^1_\theta\times M\times\rr_{h_3}$. 
 Suppose the contrary: there exists a $p\in\rr$ such that $I$ is non-constant along the hypersurface 
 $S_p=S^1\times M\times\{p\}\subset\Sigma$. 
 Then some its first order partial derivative in local coordinates of the product $S^1\times M$ is not identically zero. Let us denote the latter derivative by $g$.  Fix a point  $y\in S_0=S^1\times M\times\{0\}$ and consider 
 the  analytic extension of the function $g$ to some its complex  neighborhood $U=U(y)$ in the 
 complexified manifold $\Sigma$. Then $g\not\equiv0$ on $U$, by uniqueness of analytic extension and connectivity. On one hand, the zero locus $\{ g=0\}$ contains the sets $O_p\times\{p\}\subset\{ dI_p=0\}$. The latter sets, and hence, the intersections $\{ g=0\}\cap\{ h_3=p\}$  accumulate to all of $S_0\cap U$, as do $O_p$.

 On the other hand, the zero locus  of  a non-identically-zero analytic function $g$ on $U$ vanishing on $S_0$ 
 (where $I=const$) 
  is the union of the intersection $S_0\cap U$  and another complex hypersurface that is a closed subset in $U$ 
 intersecting the complexified hypersurface $S_0$ by an analytic subset of complex codimension two. This follows from basic analytic set theory, see 
 \cite{chirka}, which also implies that the latter codimension two analytic subset cannot contain all of the real part of the intersection 
 $S_0\cap U$. Thus, the set $\{ g=0\}\setminus S_0$ 
 cannot accumulate to the real hypersurface $S_0\cap U$.  The contradiction thus obtained proves that $I_p\equiv const$ for every $p\in\rr$ and finishes the proof of 
 Statement 5) and hence, Theorem \ref{tint}.

\section[Two-sided bounds of   balls and distance]{Two-sided bounds of the Heisenberg \\ \sR balls and distance}\label{sec:ellips}
The \sR sphere of radius $R>0$ on the Heisenberg group $G$ centered at the origin $g_0 = \Id$ is parameterized as follows \cite{ABB_book, umn1}:
\begin{align*}
&x = R \frac{\sin p}{p} \cos \tau, \quad y = R \frac{\sin p}{p} \sin \tau, \quad z = R^2 \frac{2p -\sin 2p}{8p^2}, \\
&p \in [-\pi, \pi], \quad \tau \in \R/(2 \pi \Z), 
\end{align*}
denote it as $S_R$. Each sphere is a rotation surface around the $z$-axis, and spheres of  different radii are transferred one into another by dilations
$$
\d_k(x, y, z) = (kx, ky, k^2z), \qquad k > 0, \quad (x, y, z) \in G,
$$
as follows:
\be{dilat}
\d_k(S_R) = S_{kR}.
\ee
The unit sphere $S:=S_1$ is a rotation surface around the $z$-axis of the curve
\be{S1par}
r = \frac{\sin p}{p}, \quad z = \frac{2p-\sin 2p}{8p^2}, \qquad p \in [-\pi, \pi],
\ee
where $r = \sqrt{x^2+y^2}$, see this curve in Fig. \ref{fig:S1}. The curve \eq{S1par} intersects the $z$-axis at the points $(r, z) = \left(0, \pm \frac{1}{4 \pi}\right)$. The unit sphere $S$ is shown below in coordinates $(x, y, z)$ (Fig. \ref{fig:sphxyz}) and in coordinates $(a, b, c)$ (Fig. \ref{fig:sphabc}).

\onefiglabelsizen{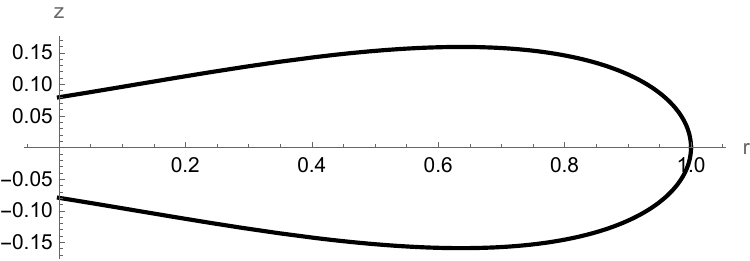}{Section of the Heisenberg unit sphere}{fig:S1}{2}

\twofiglabelsize
{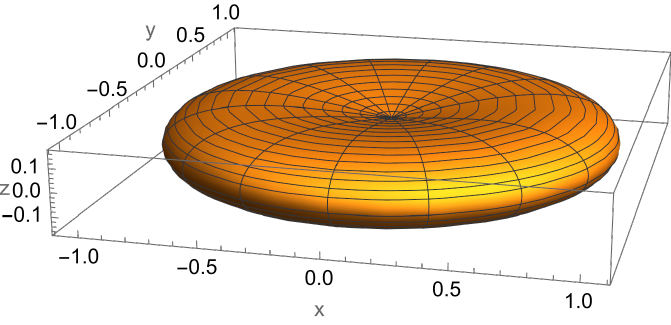}{Sphere $S$ in coordinates $(x, y, z)$}{fig:sphxyz}{0.25}
{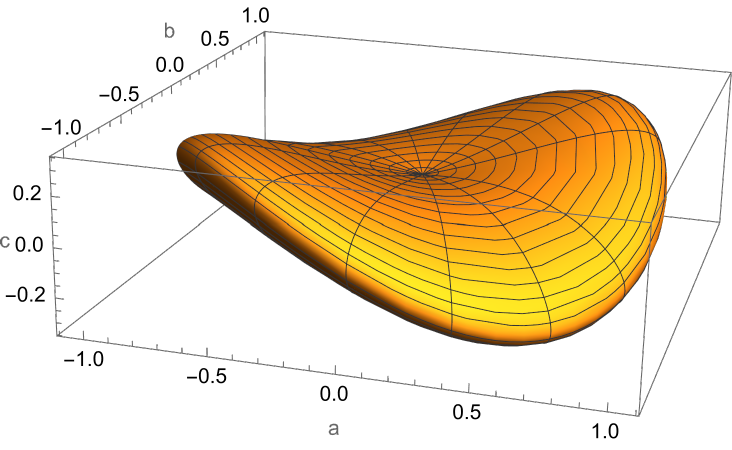}{Sphere $S$ in coordinates $(a, b, c)$}{fig:sphabc}{0.28}

Consider the following domains bounded by ellipses in the plane $\R^2_{r, z}$:
\begin{align}
&e_1 \ : \ r^2 + 16 \pi^2 z^2 \leq 1, \label{e1}\\
&e_2 \ : \ r^2 + 12  z^2 \leq 1. \label{e2}
\end{align}

\begin{lemma}\label{lem:e1}
The ellipse $\partial e_1$ passes through the points $(r, z) = (1, 0)$ and $(r, z) = \left(0, \pm \frac{1}{4 \pi}\right)$. The intersection $e_1 \cap \{r \geq 0\}$ is contained inside the curve \eq{S1par}, see Fig. $\ref{fig:S1e1}$. Moreover, the curves $\partial e_1 \cap \{r \geq 0\}$ and \eq{S1par} intersect only at the points $(r, z) = (0, \pm \frac{1}{4 \pi})$ and  $(r, z) = (1, 0)$. 
\end{lemma}
\begin{proof}
First of all, it is obvious from \eq{S1par} and \eq{e1} that the curves $\partial e_1$ and \eq{S1par} intersect   at the points $(r, z) = (0, \pm \frac{1}{4 \pi})$ and  $(r, z) = (1, 0)$.

Further, the ellipse $\partial e_1$ is the zero level curve of the function $f_1(r, z) = r^2+16\pi^2z^2-1$. Evaluation of this function on the curve \eq{S1par} is the function $\f_1(p) = \frac{\sin^2p}{p^2} + \frac{\pi^2}{4p^4}(2p-\sin 2p)^2-1$. A standard calculus shows that $\f_1(0) = \f_1(\pm \pi) = 0$, and $\f_1(p) > 0$ for $0 < |p|<\pi$, see Fig. \ref{fig:phi1(p)}.

Indeed, we have $\f_1(p) = f_1(p) + f_2(p) - 1$, $f_1(p) = \frac{\sin^2 p}{p^2}$, $f_2(p) = \frac{\pi^2}{4 p^4} (2p-\sin 2p)^2$.
Notice that $f_2'(p) = \frac{2 \pi^2}{p^5} \cos p (p \cos p - \sin p)(\sin 2p - 2p)$.

If $p \in (\pi/2, \pi)$, then $f_2'(p) < 0$, thus $f_2(p)$ decreases. Since $f_2(\pi) = 1$, then $f_2(p) > 1$, thus $\f_1(p) > 0$ for $p \in (\pi/2, \pi)$. 

If $p \in (0, \pi/2)$, then $f_2'(p) > 0$, thus $f_2(p)$ increases. Since $f_2(\pi/6) = \frac{324}{\pi^2} \left(\frac{\pi}{3} - \frac{\sqrt{3}}{2}\right)^2 \approx 1.08 >1$, then $f_2(p) > 1$  and  $\f_1(p) > 0$ for $p \in [\pi/6, \pi/2]$.

In the proof below and in next lemmas we prove bounds of the form $g_1(p) < 0$, $g_1(0) = 0$, by comparing $g_1(p)$ with appropriate and more simple function $g_2(p)$, such that $(g_1(p)/g_2(p))' g_2^2(p) > 0$. We described this method and called it ``divide et impera'' in  \cite{cartan_conj}.

We have the following equalities:
\begin{align*}
&\f_1'(p) = \frac{p \cos p - \sin p}{p^5} f_3(p),\\
&f_3(p) = 2(p^2 + \pi^2 + \pi^2\cos 2p)\sin p - 4 p \pi^2 \cos p, \\
&f_4(p) =  \left( \frac{f_3(p)}{\sin p}\right)' \sin^2 p = 2p(1+2 \pi^2-\cos 2p)+\pi^2(\sin 4p - 4 \sin 2 p),\\
&\left(\frac{f_4(p)}{\cos 2p - 2 \pi^2 - 1}\right)'(\cos 2p - 2 \pi^2 - 1)^2 = 4 f_5(p) \sin^2p,\\
&f_5(p) = -1 - 7 \pi^2 +(1+2 \pi^2+8\pi^4)\cos 2p - \pi^2 \cos 4p, \\
&f_5'(p) = 4 \cos p \sin p f_6(p), \\
&f_6(p) = - 1 - 2 \pi^2 - 8 \pi^4+4\pi^2 \cos 2p.
\end{align*}
One has $f_6(p)\leq-1+2\pi^2-8\pi^4<0$ for all $p$, since $\cos2p\leq1$. Therefore the restriction to the 
semi-interval $(0,\pi/6]$ of the function $f_5(p)$ decreases, and hence, achieves its minimum at $p=\pi/6$. 
Its  value there is equal to 
$$-1-7\pi^2+(1+2\pi^2+8\pi^4)/2+\pi^2/2=-1/2-5.5\pi^2+4\pi^4>0.$$
Therefore, $f_5(p)>0$ on the above semi-interval. Hence, the function 
$$\wt f_4(p):=\frac{f_4(p)}{\cos 2p - 2 \pi^2 - 1}$$ 
increases there and thus, 
achieves there its minimum at $p=0$. But $\wt f_4(0)=f_4(0)=0$. Therefore,  $\wt f_4(p)>0$, hence 
$f_4(p)<0$ for $p\in(0,\pi/6]$. Thus,  
the function 
$$\wt f_3(p):=\frac{f_3(p)}{\sin p}$$ 
decreases on the same semi-interval. Hence, it achieves its supremum there at $p=0$. 
But $\wt f_3(0)=0$. Therefore, $\wt f_3(p)<0$, hence $f_3(p)<0$ on the semi-interval $(0,\pi/6]$. Thus, $\f_1$ increases there, by the above formula for its derivative and since $p\cos p-\sin p<0$, i.e., 
$p<\tan p$, whenever $p\in(0,\pi/2)$. But 
$\f_1(0)=0$. Hence, $\f_1>0$ there.

Summing up, if $p \in (0, \pi)$ then $\f_1(p) > 0$. Since $\f_1(p)$ is even, this inequality holds for $0 < |p|<\pi$. 
\end{proof}

\twofiglabelsize
{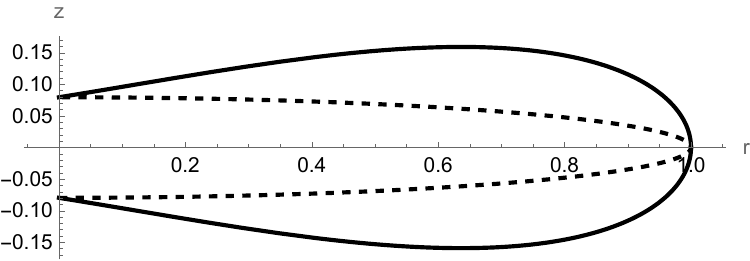}{Ellipse $\partial e_1$ inside of section of sphere $S$}{fig:S1e1}{0.18}
{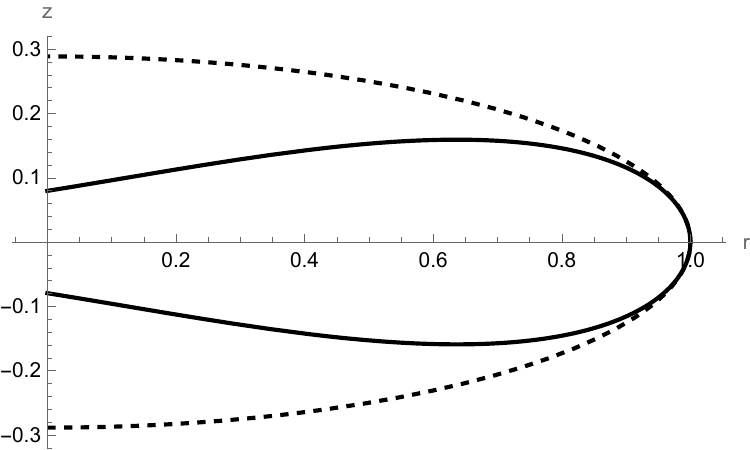}{Ellipse $\partial e_2$ outside of section of sphere $S$}{fig:S1e2}{0.27}

\twofiglabelsize
{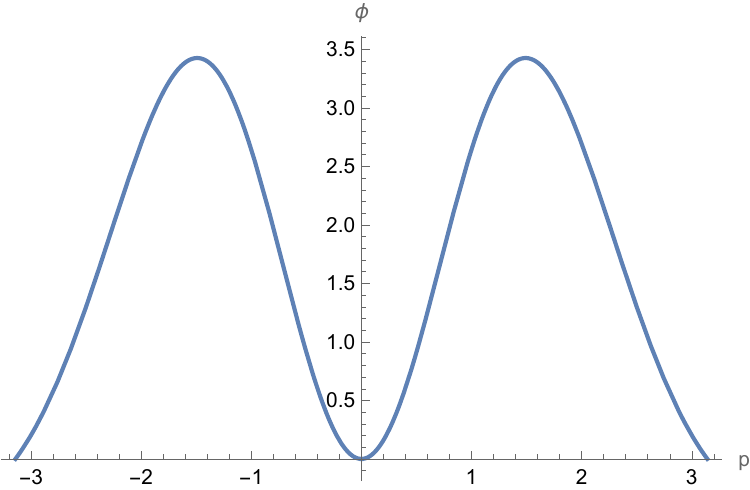}{Plot of $\f_1(p)$}{fig:phi1(p)}{0.25}
{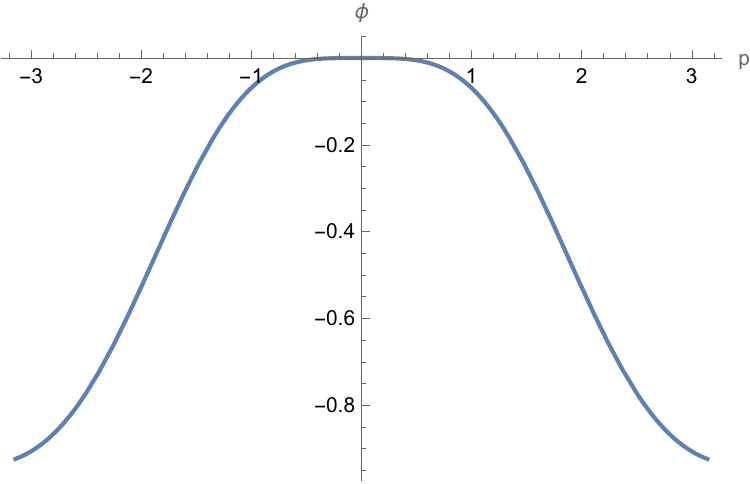}{Plot of $\f_2(p)$}{fig:phi2(p)}{0.25}

\begin{remark}
The ellipse $\partial e_1$ is the only ellipse in the plane $(r, z)$, symmetric with respect to the $z$-axis, with the properties given in Lemma \ref{lem:e1}.
\end{remark}

\begin{lemma}\label{lem:e2}
The curve $\partial e_2$ is tangent to the curve \eq{S1par} with contact of order $4$. The intersection $\partial e_2 \cap\{r \geq 0\}$ is contained outside of the curve~\eq{S1par}, see Fig. $\ref{fig:S1e2}$. Moreover, the curves $\partial e_2 \cap \{r \geq 0\}$ and \eq{S1par} intersect only at the point    $(r, z) = (1, 0)$. 
\end{lemma}
\begin{proof}
The first statement is obtained by explicit differentiation. 
Indeed,  for the ellipse $\partial e_2$ we have $r = \sqrt{1 - 12z^2} = 1 - 6z^2 + O(z^4)$, $z \to 0$. And for the curve \eq{S1par} we have $r(0) = 1$, $z(0) = 0$. In a neighbourhood of the point $(r,z)=(1,0)$, the curves in question are graphs of even functions $r(z)$. Thus, it is sufficient to prove coincidence of their second derivatives at $z=0$. One has 
\begin{align*}
&\frac{dr}{dz} = \frac{dr/dp}{dz/dp} = -\frac{2p}{\cos p}, \\
&\frac{d^2r}{dz^2} = \frac{\frac{d}{dp}(-\frac{2p}{\cos p})}{dz/dp} = \frac{4p^3(1+p \tan p)}{\cos^3 p(p - \tan p)} \to -12, \quad p \to 0,
\end{align*}
thus $r =  1 - 6z^2 + O(z^4)$, $z \to 0$.

 The second statement follows since the function $r^2+12z^2-1$ whose zero level curve is the ellipse $\partial e_2$, when restricted to the curve \eq{S1par}, gives the function $\f_2(p) = \frac{\sin^2p}{p^2} + \frac{3(2p-\sin 2p)^2}{16 p^4}-1$. A standard calculus shows that $\f_2(p) < 0$ for $0 < |p|\leq \pi$, see Fig. \ref{fig:phi2(p)}.

Indeed, we have $\f_2(p) = f_1(p) + f_2(p) - 1$, $f_1(p) = \frac{\sin^2p}{p^2}$, $f_2(p) = \frac{3(2p-\sin 2p)^2}{16 p^4}$. Further,
\begin{align*}
&f_3(p) = (16 p^4 \f_2(p))' =  40 p - 64 p^3 - 40 p \cos 2p + 4(4p^2-3)\sin 2p + 6 \sin 4p,\\ 
&f_4(p) = f_3'(p) =  8(5-24p^2 + 4(p^2-2)\cos 2p+3\cos 4p + 14 p \sin 2p), \\ 
&f_5(p) = f_4'(p) =  -16(-18p \cos 2p + (4p^2-15) \sin 2p + 6(4p + \sin 4p)), \\ 
&f_6(p) =f_5'(p) =  - 64(6+2(p^2-6)\cos 2p+6\cos 4p+11p\sin 2p), \\
&f_7(p) = \left(\frac{f_6(p)}{\cos 2p}\right)' \cos^2 2p =  32(-48 p - 4p \cos 4p + 12 \sin 2p - 11 \sin 4p + 12 \sin 6p), \\ 
&f_8(p) =f_7'(p) =  512 f_9(p) f_{10}(p), \\
&f_9(p) = \sin p - \sin 3p, \qquad f_{10}(p) = - 2 p \cos p + 3 \sin p + 9 \sin 3p, \\
&f_{11}(p) = \left(\frac{f_{10}(p)}{\cos p}\right)' \cos^2 p = 2 + 17 \cos 2p + 9 \cos 4p.
\end{align*}

Let $p \in (0, \frac{\pi}{8}]$, then $f_{11}(p) > 0$, thus $\tf_{10}(p) = \frac{f_{10}(p)}{\cos p}$ increases. Since $\tf_{10}(0) = 0$, then $\tf_{10}(p)>0$, so $f_{10}(p)>0$.

Let $p \in (\frac{\pi}{8}, \frac{\pi}{4})$. Then $-2p \cos p > -\frac{\pi}{2} \cos \frac{\pi}{4} \approx -1.11 > -2$, $3 \sin p > 3 \sin \frac{\pi}{8} \approx 1.14 > 1$, $9 \sin 3p > 9 \sin \frac{3\pi}{4} \approx 6.36 > 6$, thus $f_{10}(p) > -2 + 1 + 6 > 0$. 

Now let $p \in (0, \pi/4)$, we have proved that $f_{10}(p) > 0$. 
Since $f_9(p) < 0$, then $f_8(p) < 0$, thus $f_7(p)$ decreases.
Since $f_7(0) = 0$, then $f_7(p) < 0$, thus $\tf_6(p) = \frac{f_6(p)}{\cos 2p}$ decreases.
Since $\tf_6(0) = 0$ then $\tf_6(p) < 0$, thus $f_6(p) < 0$.
Thus $f_5(p)$ decreases, and since $f_5(0) = 0$ then $f_5(p) < 0$.
Thus $f_4(p)$ decreases, and since $f_4(0) = 0$ then $f_4(p) < 0$.
Thus $f_3(p)$ decreases, and since $f_3(0) = 0$ then $f_3(p) < 0$.
Thus $p^4 \f_2(p)$ decreases, and since $\lim_{p \to 0} p^4 \f_2(p) = 0$ then $\f_2(p) < 0$ for $p \in (0, \pi/4]$.
 
If $p \in (\pi/4, 3\pi/8)$, then $-48 p < -12 \pi < -37.6$, $|4 p \cos 4 p| < 3 \pi/2 < 4.72$, $|12 \sin 6 p|  < 12$. One has $|12 \sin 2 p - 11 \sin 4p| < 20.3$, since this is true at the endpoints $p=\pi/4, 3\pi/8$ and 
at the extremum point of the function under modulus in the interval in question. Indeed, its derivative in 
$x=2p\in(\pi/2, 3\pi/4)$ is equal to $12\cos x-22\cos2x=12u-22(2u^2-1)$, $u=\cos x$. The latter derivative vanishes, if and only if $22u^2-6u-11=0$. Solving the latter quadratic equation in negative $u=\cos x$ 
(which is indeed negative in the given interval) yields
$$u=\cos x=\frac{3-\sqrt{251}}{22}\approx-0.58377, \ \sin x=\sqrt{1-u^2}\approx\sqrt{0.65921}\approx0.811,$$
$$|12 \sin 2 p - 11 \sin 4p|=2\sin x(6-11\cos x)\approx20.169<20.3.$$
Thus, $\frac{f_7(p)}{32} < -37.6 + 4.72 + 20.3 + 12 < 0$. If $p \in [3\pi/8, \pi/2)$, then $-48 p < - 50$, $-4p \cos 4p \leq 0$, $|12 \sin 2 p - 11 \sin 4p + 12 \sin 6p| < 35$, thus $\frac{f_7(p)}{32} < -50 + 35 < 0$.

Thus for $p \in (\pi/4, \pi/2)$ we have $f_7(p) < 0$, so repeating the argument used two paragraphs above we get $f_i(p) < 0$ for $i = 3, \dots, 6$, hence $\f_2(p) < 0$.

Finally, if $p \in [\pi/2, \pi)$, then $f_1(p) \leq \frac{4}{\pi^2} < \frac 12$. 
Since $$f_2'(p) = -\frac{3}{p^5} \cos p (p \cos p - \sin p)(p-\cos p \sin p) \leq 0,$$ then $f_2(p)$ decreases, and since $f_2(\frac{\pi}{2}) = \frac{3}{\pi^2} \approx 0.3$ then $f_2(p) < \frac 12$.
Thus $\f_2(p) < 0$.

If $p = \pi$ then $\f_2(p) = \frac{3-4\pi^2}{4\pi^2} < 0$.

We proved that $\f_2(p) < 0$ for $p \in (0, \pi]$.
Since $\f_2(p)$ is even, this inequality holds for $0 < |p|\leq \pi$.
\end{proof}

\begin{remark}
The ellipse $\partial e_2$ is the smallest ellipse in the plane $(r, z)$ among ellipses symmetric with respect to the $z$-axis, tangent to the curve \eq{S1par} at the point $(r,z) = (1, 0)$ and encircling this curve.
\end{remark}

Consider the projection

\begin{align*}
&\map{P}{\R^3_{x, y, z}}{ \{(r, z) \in \R^2 \mid r \geq 0\}},\\
&P(x, y, z) = (r, z) = (\sqrt{x^2+y^2},z)
\end{align*}
and the corresponding ellipsoids $E_i = P^{-1}(e_i)$, $i = 1, 2$. Lemmas \ref{lem:e1} and~\ref{lem:e2} plus equality \eq{dilat} imply obviously the following two-sided ellipsoidal bounds of \sR balls 
$$
B_R := \{g \in G \mid d(\Id, g) \leq R\}
$$
on the Heisenberg group.

\begin{corollary}\label{cor:E12}
For any $R > 0$ we have
\be{ball_bounds}
\d_R(E_1) \subset B_R \subset \d_R(E_2).
\ee
\end{corollary} 

\begin{remark}
Estimates \eq{ball_bounds} are sharp in the sense that the ellipsoids $\d_R(E_1)$ and  $\d_R(E_2)$ are tangent to the \sR ball $B_R$ at its points in the plane $\{z = 0\}$. Moreover, the ellipsoid $\d_R(E_1)$ intersects the \sR ball $B_R$ at its points in the line $\{x = y = 0\}$, see Figs. \ref{fig:S1e1}, \ref{fig:S1e2}.
\end{remark}

\begin{remark}
In order to estimate precision of bounds \eq{ball_bounds}, take the Euclidean volume $V = dx \wedge dy \wedge dz$ (in fact, Popp's volume \cite{ABB_book}). Then 
\begin{multline*}
V(E_1) = \frac 14 < V(B_1)=\frac{1}{12} \left( 1 + 2 \pi \int_0^{2\pi} \frac{\sin x}{x} dx\right) \approx 0.83 \\
< V(E_2)= \frac{\pi}{2 \sqrt{3}} \approx 0.91.
\end{multline*}
The above integral formula for the volume $V(B_1)$ was proved in 
\cite{ball-volume}, P. 587.
\end{remark}

\begin{corollary}\label{cor:udod}
Let $g = (x, y, z) \in G$, and let $r = \sqrt{x^2+y^2}$. Then
\be{dleq}
\ud(g) := \sqrt{\frac{\sqrt{r^4+48z^2}+r^2}{2}} \leq d(\Id, g) \leq \sqrt{\frac{\sqrt{r^4+64 \pi^2 z^2}+r^2}{2}} =: \od(g).
\ee
\end{corollary}
\begin{proof}
Since the statement holds trivially for $g = \Id = (0, 0, 0)$, we can assume that $g \neq \Id$, then $R := d(\Id, g) > 0$. Denote $g'= \d_{\frac 1R} (g)$, then $d(\Id, g') = 1$, and inclusions \eq{ball_bounds} imply that the functions $f_1(x, y, z) = r^2 + 16 \pi^2z^2$ and $f_2(x, y, z) = r^2 + 12z^2$, $r = \sqrt{x^2 + y^2}$, satisfy the inequalities $$f_1(g') = f_1\left(\frac xR, \frac yR, \frac{z}{R^2}\right) \geq 1 = d(\Id, g') \geq f_2\left(\frac xR, \frac yR, \frac{z}{R^2}\right) = f_2(g').$$  
Thus $\frac{r^2}{R^2} + 16\pi^2\frac{z^2}{R^4} \geq 1 \geq \frac{r^2}{R^2} + 12\frac{z^2}{R^4}$, i.e., 
\be{al}
16 \pi^2 z^2\a^2 + r^2 \a \geq 1 \geq 12 z^2\a^2 +r^2\a, \qquad \a = \frac{1}{R^2}.
\ee
  The second  inequality in \eq{al} solves to $0 < \a \leq \a_2 := 
	\frac{2}{\sqrt{r^4+48z^2}+r^2}$, whence 
	$R \geq \frac{1}{\sqrt{\a_2}} = \sqrt{\frac{\sqrt{r^4+48z^2}+r^2}{2}}$, which gives the first inequality in \eq{dleq}. 
	Similarly, the first inequality in \eq{al} solves to $\a \geq \a_1 := \frac{2}{\sqrt{r^4+64 \pi^2 z^2}+r^2}$, whence 
	$R \leq \frac{1}{\sqrt{\a_1}} = \sqrt{\frac{\sqrt{r^4+64 \pi^2 z^2} + r^2}{2}}$, which gives the second inequality in \eq{dleq}.
\end{proof}

\begin{remark}
Estimates \eq{dleq} are functional expressions of bounds \eq{ball_bounds}: 
$$
\{g \in G \mid \ud(g) \leq R \} = \d_R(E_2), \qquad \{g \in G \mid \od(g) \leq R \} = \d_R(E_1).
$$
\end{remark}

\begin{remark}
Estimates \eq{dleq} are sharp in the following sense:
\begin{enumerate}
\item
in the case $z = 0$ these inequalities turn into equalities,
\item
in the case $r = 0$ the second inequality turns into equality corresponding to $e_1$. 
\end{enumerate}
In the case $rz\neq 0$ the both inequalities \eq{dleq} are strict.
\end{remark}

The second inclusion in \eq{ball_bounds} obviously implies the following inclusions.

\begin{corollary}
For any $R > 0$ we have
\begin{align*}
&B_R \subset \left\{ g = (x, y, z) \in G \mid \sqrt{x^2 + y^2} \leq R\right\},\\
&B_R \subset \left\{ g = (x, y, z) \in G \mid |z| \leq \frac{R^2}{\sqrt{12}}\right\},
\end{align*}
or, which is equivalent,
\begin{align}
&d(\Id, g) \geq \sqrt{x^2 + y^2}, \label{d>=1}\\
&d(\Id, g) \geq \sqrt[4]{12z^2}. \label{d>=2}
\end{align}
\end{corollary} 

Notice that inequalities \eq{d>=1}, \eq{d>=2} follow also from the first inequality in \eq{dleq}.

\section[Bounds of cut time via lower bounds of \sR balls]{Bounds of cut time via lower bounds \\of \sR balls}\label{sec:lower}

 Fix a point $q_0=g_0'\in M$.
Denote the ball $B'_t = \{q \in M \mid d'(q_0, q) \leq t\}$, $t \geq 0$, where  $d'$ is the \sR distance on $M$.
Denote also 
$$
\tb = \inf \{t > 0 \mid B'_t= M\}.
$$

The following lemmas show the relevance of the number $\tb$ for the \sR manifold $M$.

\begin{lemma}\label{lem:tbard}
We have the following:
\begin{itemize}
\item[$(1)$]
$\tb = \sup \{ d'(q_0, q_1) \mid q_1 \in M \}$.
\item[$(2)$]
$\tb = \sup \{ \tcut(q(\cdot)) \mid q(\cdot) \subset M  \text{ a geodesic s.t. } q(0) = q_0\}$.
\end{itemize}
\end{lemma}
\begin{proof}
(1) Denote $t_1 = \sup \{ d'(q_0, q_1) \mid q_1 \in M \}$ and assume by contradiction that $\tb \neq t_1$.

Let $t_1 < \tb$. Then for every $t \in (t_1, \tb)$ and every $q_1 \in M$ we have $d'(q_0, q_1)<t$, i.e., $q_1 \in B'_t$. Since $t<\tb$, this contradicts to definition of~$\tb$.

Let $t_1 > \tb$. Then for every $t \in (\tb, t_1)$ there exists $q_1 \in M$ such that $d'(q_0, q_1) > t$, i.e., $q_1 \notin B'_t$.  Since $t>\tb$, this contradicts to definition of $\tb$ once more.

\medskip

(2) Denote $t_2 = \sup \{ \tcut(q(\cdot)) \mid q(\cdot) \subset M  \text{ a geodesic s.t. } q(0) = q_0\}$ and assume by contradiction that $\tb \neq t_2$.

Let $t_2<\tb$, take any $t \in (t_2,\tb)$. Then $B'_t \neq M$, thus there exists a point $q_1 \in M$ such that $d'(q_0, q_1) > t$. Take a \sR length minimizer $q(\cdot)$ connecting $q_0$ and $q_1$. We have $\tcut(q(\cdot)) > t > t_2$, which contradicts the definition of $t_2$.

Let $\tb < t_2$, take any $t \in (\tb, t_2)$. We have $B'_t = M$, thus for every $q_1\in M$ one has the inequality  $d'(q_0, q_1) \leq t$. Then for every geodesic $q(\cdot) \subset M$ starting at $q_0$ we have $\tcut(q(\cdot)) \leq t < t_2$, which contradicts the definition of~$t_2$.
\end{proof}

\begin{remark}
Consider the diameter of the \sR manifold $M$:
$$
\diam(M) = \sup \{d'(q_1,q_2)\mid q_1, q_2 \in M\}. 
$$
By the triangle inequality, we have a bound $\diam(M) \leq 2 \tb$.
\end{remark}

\begin{theorem}
We have $\tb \leq \ts := \frac 12 \sqrt{\frac 12 \left(1 + \sqrt{1 + 1024 \pi^2}\right)} \approx 3.56$.
\end{theorem}
\begin{proof}
We  show that $B'_{\ts} = M$.
By  Corollary \ref{cor:E12}, we have $B_{\ts} \supset \d_{\ts}(E_1) =:E_{1\ts}$, where the ellipsoid $E_{1\ts}\subset G$ is defined by the inequality $\frac{r^2}{{\tsd}} + \frac{16 \pi^2 z^2}{{\tsf}} \leq 1$. Thus $B'_{\ts} \supset E'_{1{\ts}} := \pi(E_{1{\ts}})$. We show that $E'_{1{\ts}} = M$.

The homogeneous space $M$ can be represented by a fundamental domain $D = \{(a, b, c) \in G \mid 0 \leq a, b, c<1\}$, so that $\pi(D) = M$. We have $D \subset \cup_{i=1}^8 K_i$, where the cubes $K_i$ are defined as follows:
\begin{align*}
&K_1 \ : \ 0 \leq a, b, c \leq \frac 12,\qquad
&K_2 \ : \ 0 \leq a, c \leq \frac 12 \leq b \leq 1,\\
&K_3 \ : \ 0 \leq a \leq \frac 12\leq b, c \leq 1,\qquad
&K_4 \ : \ 0 \leq a, b \leq \frac 12\leq c \leq 1,\\
&K_5 \ : \ 0 \leq  b, c \leq \frac 12\leq a \leq 1,\qquad
&K_6 \ : \ 0 \leq   c \leq \frac 12\leq a,b \leq 1,\\
&K_7 \ : \ \frac 12 \leq a, b, c \leq 1,\qquad
&K_8 \ : \ 0 \leq   b \leq \frac 12\leq a,c \leq 1.
\end{align*}
We show that $E'_{1{\ts}} \supset \pi(K_i)$, $i = 1, \dots, 8$, which implies that $E'_{1{\ts}} = M$. To this end we define the following points $g_i \in H$ in the coordinates $(a, b, c)$: $g_1 = (0, 0, 0)$,  $g_2 = (0, 1, 0)$, $g_3 = (0, 1, 1)$, $g_4 = (0, 0, 1)$, $g_5 = (1, 0, 0)$, $g_6 = (1, 1, 0)$, $g_7 = (1, 1, 1)$, $g_8 = (1, 0, 1)$,  and prove that $E_{1\ts} \supset \tK_i := g_i^{-1}K_i$, $i = 1, \dots, 8$.

Let $L \subset \R^n$ be a convex compact set. We call a continuous function $\map{f}{L}{\R}$ {\it quasiconvex} if $\max_L f = \max_{\partial L} f$. Since a convex function on a convex compact set attains maximum at points of the boundary of this set or at all points of this set, then a convex function on such a set is quasiconvex.

Now let $\Pi \subset \R^3$ be a parallelepiped whose all faces and edges are parallel to coordinate planes and axes, and let $\dim \Pi \in \{1, 2, 3\}$, i.e., $\Pi$ is a 3D~parallelepiped, a 2D rectangle, or a 1D segment. Let us study quasiconvexity of the function $f_t(a, b, c) = t^2(a^2+b^2) + 4 \pi^2 (2c-ab)^2 - t^4$ whose zero level is the ellipsoid $\partial E_{1t}$, $t > 0$. We have 
$\pder{f_t}{a} = 2 t^2 a - 8 \pi^2b(2c-ab)$,  $\pder{f_t}{b} = 2 t^2 b - 8 \pi^2a(2c-ab)$,
$\pder{f_t}{c} = 16 \pi^2(2c-ab)$, thus $f_t$ has only one critical point $(a, b, c) = (0, 0, 0)$, which is the minimum point. Thus if $\dim \Pi = 3$ then $\restr{f_t}{\Pi}$ is quasiconvex.

If $\Pi \subset \{a = \const\}$ or $\Pi \subset \{b = \const\}$, then $\restr{f_t}{\Pi}$ is convex, thus it is quasiconvex. Thus if $\dim \Pi = 3$ and the restriction of $f_t$ to faces of $\Pi$ parallel to the plane $\{c = 0\}$ is quasiconvex, then $\restr{f_t}{\Pi}$ attains maximum at vertices of $\Pi$.

1) We prove that $E_{1\ts} \supset \tK_1 = K_1$. Since 
$\restr{\pder{f_t}{a}}{c = 0} = {2a(4\pi^2b^2+t^2)}$, which is nonnegative and vanishes only for $a = 0$, then 
the function \\
$\restr{f_{\ts}}{K_1 \cap \{c=0\}}$ increases in $a$, thus 
$\restr{f_t}{K_1 \cap \{c=0\}}$ is quasiconvex.

We have $p := \left(\restr{\pder{f_{\ts}}{a} + \pder{f_{\ts}}{b}\right)}{K_1 \cap \{c=1/2\}} = 2(a+b)(\tsd-4(1-ab)\pi^2)$. Since in $K_1$ we have $ab\leq \frac 14 < 1 - \frac{\tsd}{4\pi^2} \approx 0.68$, then $p $ is nonpositive and vanishes only at $(a, b) = (0, 0)$, then 
$\restr{f_{\ts}}{K_1 \cap \{c=1/2\}} $ is quasiconvex. 

Thus $\restr{f_{\ts}}{K_1}$ attains maximum at vertices of $K_1$. We have 
$f_{\ts}(0, 0, 0) \approx - 161$, $f_{\ts}(0, 0, 1/2) \approx - 122$, $f_{\ts}(0, 1/2, 0) = f_{\ts}( 1/2, 0, 0) \approx - 158$, \\
$f_{\ts}(0, 1/2, 1/2) = f_{\ts}( 1/2, 0, 1/2) \approx - 118$, $f_{\ts}(1/2, 1/2, 0) \approx - 152$, \\
$f_{\ts}(1/2, 1/2, 1/2) \approx - 133$,
whence $\restr{f_{\ts}}{K_1} < 0$, thus $E_{1\ts} \supset \tK_1 = K_1$. 

2) We prove that $E_{1\ts} \supset \tK_2 =  g_2^{-1} K_2$. 
Notice that for any elements $(\alpha, \beta, \gamma)$, $(a, b, c)$ of $G$
$$
(\alpha, \beta, \gamma)^{-1} \cdot (a, b, c) = (a - \alpha, b - \beta, c - \gamma + \alpha(\beta - b)).
$$
Thus 
$\tK_2  = \{a, c \in [0, 1/2], \ b \in [-1/2, 0]\}$. By the argument of item 1), the function $\restr{f_{\ts}}{\tK_2 \cap \{c=0\}}$ increases in $a$, thus $\restr{f_t}{\tK_2 \cap \{c=0\}}$ is quasiconvex.

We have $\restr{\pder{f_t}{a}}{\tK_2 \cap \{c=1/2\}} = 8b(ab-1)\pi^2 + 2 at^2$, which is nonnegative and vanishes only for $a = b = 0$, thus $\restr{f_t}{\tK_2 \cap \{c=1/2\}}$ is quasiconvex. 

So $\restr{f_{\ts}}{\tK_2}$ attains maximum at vertices of $\tK_2$. In item 1) we showed that $f_{\ts}<0$ at vertices of the square $[0, 1/2]_a \times \{b = 0\} \times [0, 1/2]_c$. Further, we have 
$f_{\ts}(0, -1/2, 0) \approx - 158$, $f_{\ts}(0, -1/2, 1/2) \approx - 118$, $f_{\ts}(1/2, -1/2, 0) \approx - 152$, $f_{\ts}(1/2, -1/2, 1/2) \approx - 93$, then $\restr{f_{\ts}}{\tK_2} < 0$, thus $E_{1\ts} \supset \tK_2$.

3) We prove that $E_{1\ts} \supset \tK_3 =  g_3^{-1} K_3 = \{a \in [0, 1/2], \ b, c \in [-1/2, 0]\}$. Consider the involution $\mapto{i}{(a, b, c)}{(a, -b, -c)}$. Then $i(\tK_1) = \tK_3$ and $f_t \circ i = f_t$, thus $\restr{f_{\ts}}{\tK_3}  < 0$ since 
$\restr{f_{\ts}\circ i }{\tK_1} = \restr{f_{\ts}}{\tK_1} < 0$. Thus $E_{1\ts} \supset \tK_3$.

4) We prove that $E_{1\ts} \supset \tK_4 =  g_4^{-1} K_4 = \{a, b \in [0, 1/2], \ c \in [-1/2, 0]\}$.
Consider the involution $\mapto{i}{(a, b, c)}{(a, -b, -c)}$. Then $i(\tK_2) = \tK_4$ and $f_t \circ i = f_t$, thus $\restr{f_{\ts}}{\tK_4}  < 0$ by virtue of $\restr{f_{\ts}}{\tK_2}  < 0$ similarly to item 3).

5) We prove that $E_{1\ts} \supset \tK_5 =  g_5^{-1} K_5$. We have $\tK_5 \subset \hK_5  = \{a \in [-1/2, 0], \ b \in [0, 1/2], \  c \in [-1/2, 1/2]\}$.
Notice that  $\restr{\pder{f_t}{a}}{c = -1/2} = 2a(t^2 + 4 \pi^2 b^2)+8\pi^2b = 0$ only if 
$a = \bar a := - \frac{4\pi^2b}{t^2+4\pi^2b^2}$, and 
$$\restr{\pder{f_{\ts}}{b}}{a = \bar a, \ b \in [0,1/2], \ c = -1/2} = 
\frac{2b\tsd(16\pi^4(b^4-1)+8b^2\pi^2\tsd + \tsf)}{(4 b^2\pi^2+\tsd)^2},$$ which is nonpositive (since the quartic polynomial in brackets in numerator is negative for $b\in [0,1/2]$) and vanishes only for $b = 0$. Thus $\restr{f_{\ts}}{\hK^5 \cap \{c = -1/2\}}$ has no interior critical points, so this function is quasiconvex.

We have   $\restr{\pder{f_t}{a}}{\hK^5 \cap \{c = 1/2\}} = 2a(t^2 + 4 \pi^2 b^2)-8\pi^2b = 0$ only if 
$a = \bar a :=  \frac{4\pi^2b}{t^2+4\pi^2b^2}$, and $\restr{\pder{f_{\ts}}{b}}{a = \bar a, \ b \in [0,1/2], \ c = 1/2} = 
\frac{2b\tsd(16\pi^4(b^4-1)+8b^2\pi^2\tsd + \tsf)}{(4 b^2\pi^2+\tsd)^2}$, which is nonpositive  and vanishes only for $b = 0$ by the previous paragraph. Thus $\restr{f_{\ts}}{\hK_5 \cap \{c = 1/2\}}$ has no interior critical points, so this function is quasiconvex.

So $\restr{f_{\ts}}{\hK_5}$ attains maximum at vertices of $\hK_5$. Since $f_{\ts}(-1/2, 0, -1/2)=f_{\ts}(-1/2, 0, 1/2)=f_{\ts}(1/2, 0, 1/2)=f_{\ts}(0,1/2,1/2)=f_{\ts}(0,1/2,-1/2)\approx -118$, 
$f_{\ts}(-1/2, 1/2, -1/2)=f_{\ts}(1/2, 1/2, 1/2)\approx -132$, $f_{\ts}(-1/2, 1/2, 1/2)=f_{\ts}(1/2, -1/2, 1/2)
\approx-93$,   $f_{\ts}(0, 0, -1/2)=f_{\ts}(0,0,1/2)\approx -121$, see items 1)--3) above, then $\restr{f_{\ts}}{\hK_5} \leq 0$, thus $\restr{f_{\ts}}{\tK_5} \leq 0$, and $E_{1\ts} \supset \tK_5$.

6) We prove that $E_{1\ts} \supset \tK_6 =  g_6^{-1} K_6$. We have $\tK_6 \subset \hK_6  = \{a, b \in [-1/2, 0],  \  c \in [0, 1]\}$.
Since $\restr{\pder{f_t}{a}}{\hK_6 \cap \{c = 0\}} = 2a(4b^2\pi^2+t^2)$, which is nonpositive and vanishes only for $a = 0$, then $\restr{f_t}{\hK_6 \cap \{c = 0\}}$ is quasiconvex. 

We have $\delta := 
\restr{\left(\pder{f_t}{a} + \pder{f_t}{b}\right)}{\hK_6 \cap \{c = 1\}} = 2(a+b)(t^2-4\pi^2 (2-ab))$. Since $\frac{\tsd}{4\pi^2} \approx 0.32 < \frac 74 \leq 2 - ab$, then $\delta$ is nonnegative and vanishes only for $a = b = 0$. Thus $\restr{f_t}{\hK_6 \cap \{c = 1\}}$ is quasiconvex. So $\restr{f_{\ts}}{\tK_6}$ attains maximum at vertices of $\tK_6$.

Since $f_{\ts}(-1/2, -1/2, 0)=f_{\ts}(1/2, 1/2, 0)\approx -152$, see item 1), \\ $f_{\ts}(-1/2, -1/2, 1)\approx -34$, \\
$f_{\ts}(-1/2, 0, 0) = f_{\ts}(0, -1/2, 0) \approx -158$,  see item 1), $f_{\ts}(-1/2, 0, 1) = f_{\ts}(0, -1/2, 1) =0$, $f_{\ts}(0, 0, 0)\approx -161$, see item 1), $f_{\ts}(0, 0, 1)\approx -3$,
then $\restr{f_{\ts}}{\hK^6} \leq 0$, thus $\restr{f_{\ts}}{\tK^6} \leq 0$, and $E_{1\ts} \supset \tK_6$.

7) We prove that $E_{1\ts} \supset \tK_7 =  g_7^{-1} K_7$. We have $\tK_7 \subset \hK_7  = \{a, b \in [-1/2, 0],  \  c \in [-1/2, 1/2]\}$.
Consider the involution $\map{i}{(a, b, c)}{(a, -b, -c)}$. Then $i(\hK_5) = \hK_7$ and $f_t \circ i = f_t$. Since  $\restr{f_{\ts}}{\tK^5} \leq 0$, then  $\restr{f_{\ts}}{\tK^7} \leq 0$,  and $E_{1\ts} \supset \tK_7$.

8) Finally,  we prove that $E_{1\ts} \supset \tK_8 =  g_8^{-1} K_8$. We have $\tK_8 \subset \hK_8  = \{a \in [-1/2, 0], \ b \in [0, 1/2],  \  c \in [-1, 0]\}$.
Consider the involution $\map{i}{(a, b, c)}{(a, -b, -c)}$. Then $i(\hK_6) = \hK_8$ and $f_t \circ i = f_t$. Since  $\restr{f_{\ts}}{\tK^6} \leq 0$, then  $\restr{f_{\ts}}{\tK^8} \leq 0$,  and $E_{1\ts} \supset \tK_8$.

Summing up, we proved that $E_{1\ts} \supset \cup_{i=1}^8 g_i^{-1}K_i$.
Thus $$E'_{1\ts} \supset \cup_{i=1}^8 \pi(g_i^{-1}(K_i)) \supset \pi(D) = M,$$ so the required inclusion $B'_{\ts} = M$ follows.
\end{proof}

We plot a union of \sR balls $B_{\ts}(h)$ for some $h \in H$ covering the fundamental domain $D$ in Fig. \ref{fig:tbar}.

\onefiglabelsizen{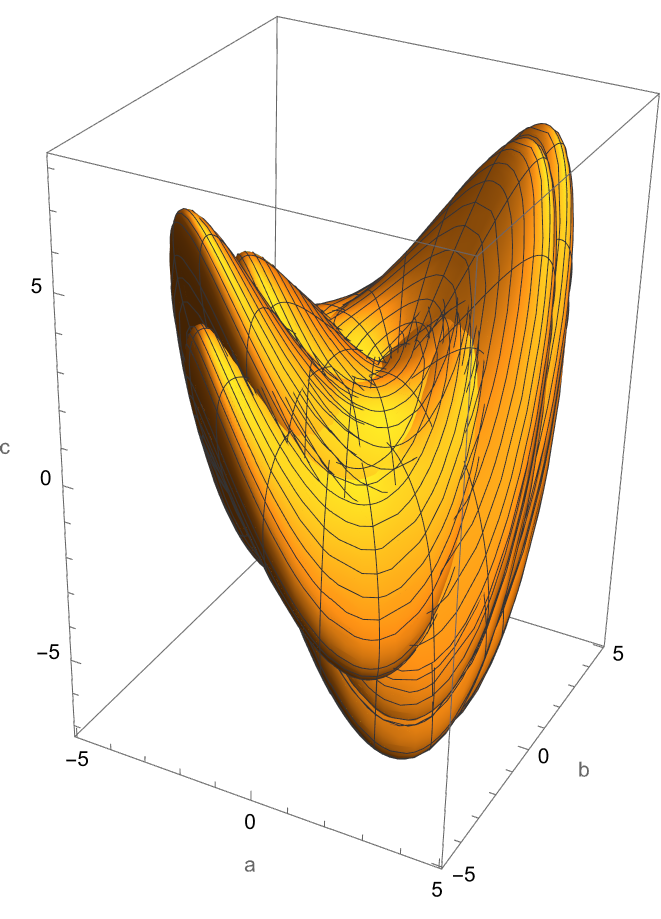}{Union of sub-Riemannian 
Heisenberg balls $B_{t}(h)$ covering the fundamental domain $D$}{fig:tbar}{7}

Now we provide a lower bound of the number $\tb$. Denote the points $\bg = (\ba, \bb, \bc), \tg = (\ta, \tib, \tc) \in G$ such that $\ba = \bb = \bc = \frac 12$ and $\ta = 1$, $\tib = \tc = 0$.

\begin{theorem}
We have $\tb \geq d(\tg, \bg)$.
\end{theorem} 
\begin{proof}
By left-invariance of the metric $d$, for any $g_1, g_2 \in G$ we have $d(g_1, g_2) = d(\Id, g_0)$, $g_0 = g_1^{-1}g_2$. The distance $d(\Id, g_0)$, $g_0 = (x, y, z)$, is computed explicitly \cite{umn1}: if $(x, y) \neq (0, 0)$ then
\begin{align}
&d(\Id, g_0) = \frac{p}{\sin p} \sqrt{x^2+y^2}, \label{dist1}\\
&\frac{2p-\sin 2p}{8 \sin^2 p} = \frac{|z|}{x^2+y^2}, \qquad p \in [0, \pi], \label{dist2}
\end{align}
and if $(x, y) = (0, 0)$ then $d(\Id, g_0) = 2 \sqrt{\pi|z|}$.

Now we show that $d(\tg, \bg) < 1$. We have $d(\tg, \bg) = d(\Id, g)$, where $g = \tg^{-1}\bg = (x, y, z) = (-1/2, 1/2, 1/8)$. The function $\psi(p) := \frac{2p-\sin 2p}{8 \sin^2 p}$ that appears in \eq{dist2} increases as $p \in (0, \pi)$ from $0$ to $+\infty$. Indeed, changing $p$ to $u=2p$, we get 
$$4(1-\cos u)^2\frac{d\psi}{du}=(1-\cos u)^2-(u-\sin u)\sin u=2(1-\cos u)-u\sin u$$
$$=4\sin p(\sin p-p\cos p)>0 \ \text{ for } \ p\in(0,\pi).$$ 
Since $\psi(1.25) \approx 0.26$ then $\psi(1.25) > \frac 14$. Let $\ps \in (0, \pi)$ be the root of the equation $\psi(\ps) = \frac 14 = \frac{|z|}{x^2+y^2}$. Then $\ps = \psi^{-1}(\frac 14) < 1.25$, thus by equalities \eq{dist1}, \eq{dist2} 
$$
d(\Id, \tg^{-1} \bg) = \frac{\ps}{\sin \ps} \sqrt{\left(\frac 12\right)^2 + \left(\frac 12\right)^2} < \frac{1.25}{\sin 1.25} \frac{1}{\sqrt{2}}\approx 0.93 < 1. 
$$
Here we use increasing of the function $\frac p{\sin p}$ in $p\in(0,\pi)$. Indeed, its derivative multiplied by 
$\sin^2p$ is equal to $\sin p-p\cos p>0$.  So $d(\tg, \bg) < 1$.

Take any number $\tit \in (0, d(\tg, \bg))$. We show that $B'_{\tit} \neq M$. To this end we show that $\pi(\bg) \notin B'_{\tit}$, i.e., that $\bg \notin H B_{\tit}$.

We take any $h = (a, b, c) \in H$ and show that $\bg \notin h B_{\tit}$.
The following cases are possible:

1) $0 \leq a, b, c \leq 1$,

2) $(a \leq -1) \quad \lor \quad (a \geq 2) \quad \lor \quad (b \leq -1) \quad \lor \quad (b \geq 2)$,

3) $(0 \leq a, b \leq 1) \quad \land \quad ((c \leq -1) \quad \lor \quad (c \geq 2))$. 

\medskip
1) Let $0 \leq a, b, c \leq 1$.

1.1) Let $h = (0, 0, 0)$ in coordinates $a, b, c$, we denote this as $h = (0, 0, 0)_{abc}$. 
Then $g_0 = h^{-1}\bg = \bg = (1/2, 1/2, 1/2)_{abc} = (1/2, 1/2, 3/8)_{xyz}$, i.e., the point $\bg$ has coordinates $(x, y, z) = (1/2, 1/2, 3/8)$. 
Thus by Corollary \ref{cor:udod}  
$d(\Id, g_0) \geq \ud(g_0) = \sqrt{\frac 12 \left(\frac 12 + \sqrt{7}\right)} \approx 1.25 > 1 > d(\tg, \bg) > \tit$.

1.2) Let $h = (0, 0, 1)_{abc}$.
Then $g_0 = h^{-1}\bg  = (1/2, 1/2, -1/2)_{abc} = (1/2, 1/2, -5/8)_{xyz}$, thus by Corollary \ref{cor:udod}\\
$d(\Id, g_0) \geq \ud(g_0) = \sqrt{\frac 12 \left(\frac 12 + \sqrt{19}\right)} \approx 1.56 > 1 > d(\tg, \bg) > \tit$.

1.3) Let $h = (0, 1, 0)_{abc}$.\\
Then $g_0 = h^{-1}\bg  = (1/2, -1/2, 1/2)_{abc} = (1/2, -1/2, 5/8)_{xyz}$, thus by Corollary \ref{cor:udod}\\
$d(\Id, g_0) \geq \ud(g_0) = \sqrt{\frac 12 \left(\frac 12 + \sqrt{19}\right)} \approx 1.56 > 1 > d(\tg, \bg) > \tit$.

1.4) Let $h = (0, 1, 1)_{abc}$.
Then $g_0 = h^{-1}\bg  = (1/2, -1/2, -1/2)_{abc} = (1/2, -1/2, -3/8)_{xyz}$, thus by Corollary \ref{cor:udod}\\
$d(\Id, g_0) \geq \ud(g_0) = \sqrt{\frac 12 \left(\frac 12 + \sqrt{7}\right)} \approx 1.25 > 1 > d(\tg, \bg) > \tit$.

1.5) Let $h = (1, 0, 0)_{abc}$. Then $h = \tg$, thus $d(h, \bg) = d(\tg, \bg) 
> \tit$.

1.6) Let $h = (1, 0, 1)_{abc}$.
Then $g_0 = h^{-1}\bg  = (-1/2, 1/2, -1)_{abc} = (-1/2, 1/2, -7/8)_{xyz}$, thus by Corollary \ref{cor:udod}\\
$d(\Id, g_0) \geq \ud(g_0) = \sqrt{\frac 12 \left(\frac 12 + \sqrt{37}\right)} \approx 1.81 > 1 > d(\tg, \bg) > \tit$.

1.7) Let $h = (1, 1, 0)_{abc}$.
Then $g_0 = h^{-1}\bg  = (-1/2, -1/2, 1)_{abc} = (-1/2, -1/2, 7/8)_{xyz}$, thus by Corollary \ref{cor:udod}\\
$d(\Id, g_0) \geq \ud(g_0) = \sqrt{\frac 12 \left(\frac 12 + \sqrt{37}\right)} \approx 1.81 > 1 > d(\tg, \bg) > \tit$.

1.8) Let $h = (1, 1, 1)_{abc}$.
Then $g_0 = h^{-1}\bg  = (-1/2, -1/2, 0)_{abc} = (-1/2, -1/2, -1/8)_{xyz}$.
Consider the involution $\map{i}{(x, y, z)}{(x, -y, -z)}$, then $d(\Id, i(g)) = d(\Id, g)$. Since $i(g_0) = \tg^{-1}\bg$, then $d(\Id, g_0) = d(\Id, \tg^{-1}\bg) > \tit$.

2) Let $(a \leq -1) \quad \lor \quad (a \geq 2) \quad \lor \quad (b \leq -1) \quad \lor \quad (b \geq 2)$.
Since $d(h, \bg) = d(\Id, h^{-1}\bg)$ and $h^{-1}\bg = (\frac 12 -a, \frac 12 - b, *)$, then by inequality \eq{d>=1}
$$
d(h, \bg) \geq \sqrt{\left(\frac 12 - a\right)^2 + \left(\frac12 -b\right)^2} \geq \sqrt{\left(\frac 32\right)^2 + \left(\frac 12\right)^2} = \frac{\sqrt{10}}{2} \approx 1.58.
$$
Thus $d(h, \bg) > d(\tg, \bg) > \tit$.

3) Let $(0 \leq a, b \leq 1) \quad \land \quad ((c \leq -1) \quad \lor \quad (c \geq 2))$.
 We have $d(h, \bg) = d(\Id, g_0)$,
$$
g_0 = h^{-1}\bg = (x_0, y_0, z_0) = \left( \frac 12 - a, \frac 12 - b, \frac{ab}{2} + \frac 38 + \frac{b-a}{4} - c\right).
$$
If $c \geq 2$, then $|z_0| \geq \frac 78$. And if $c \leq -1$, then $|z_0| \geq \frac 98$. In both cases inequality~\eq{d>=2} implies that $d(h, \bg) = d(\Id, h^{-1}\bg) \geq \restr{\sqrt[4]{12z_0^2}}{|z_0| = \frac 78} = \sqrt[4]{\frac{147}{16}} >1>\tit$.

Summing up, we proved that $\bg \notin H B_{\tit}$, and the statement of this theorem follows.
\end{proof}

\begin{remark}
Numerical computations on the basis of
equalities \eq{dist1}, \eq{dist2} imply that $d(\tig, \bg) \in (0.91, 0.92)$. 
\end{remark}

\begin{remark}
For comparison, consider the standard Euclidean metric on~$\R^3$ and its quotient on the torus $\T^3 = \R^3/\Z^3$ (see Sec. \ref{sec:ex}). Then formula \eq{tcut} yields
$$
\sup \{ \tcut(q(\cdot)) \mid q(\cdot) \text{ geodesic on } \T^3\} = \frac{\sqrt{3}}{2} \approx 0.87.
$$
This value is essentially less than our bound $\tb \leq \ts \approx 3.56$ since in the Heisenberg group the \sR distance grows slowly near the origin in the direction of the vector field $X_3$, see Fig. \ref{fig:tbar} and estimates \eq{dleq}.
\end{remark}

\section[Bounds of cut time via upper bounds of \sR balls]{Bounds of cut time via upper bounds \\of \sR balls}\label{sec:upper}

Recall that $B_t(g) \subset G$ is the closed \sR ball of radius $t \geq 0$ centered at a point $g \in G$, and $B_t := B_t(\Id)$.
Denote 
$$
\that = \sup \{ t > 0 \mid B_t(h_1) \cap B_t(h_2) = \emptyset \ \  \forall h_1 \neq h_2 \in H\}.
$$
Since $B_t(h_i) = h_iB_t$, then 
$$
\that = \sup \{ t > 0 \mid B_t \cap B_t(h) = \emptyset \ \ \forall h \neq \Id \in H\}.
$$
Recall that $g_0 = \Id \in G$, and for an element $g \in G$ we denote its projection to $M$ as $g':=\pi(g)$.

\begin{lemma}\label{lem:dd'1}
If $d(g_0, g_1) < \that$ for an element $g_1 \in G$, then $d'(g_0', g_1') = d(g_0, g_1)$. 
\end{lemma}
\begin{proof}
Let $d(g_0, g_1) < \that$. Notice that
\be{d'mind}
d'(g_0', g_1') = \min\{d(\bar{g}_0, \bar{g}_1) \mid \pi(\bar{g}_i) = g_i'\} = d(\tig_0, \tig_1) 
\ee
for some $\tig_i\in\pi^{-1}(g'_i)$. 
We have $\tig_i = h_ig_i$, $h_i \in H$. 

If $h_0 = h_1$, then $d'(g_0', g_1') = d(h_0g_0, h_0g_1) = d (g_0, g_1)$, and the claim follows.

Let $h_0\neq h_1$. Then $d(\tig_0, \tig_1) = d(h_0g_0, h_1g_1) = d(g_0, h_0^{-1}h_1g_1)$. Moreover, $h_0^{-1}h_1g_1 \in h_0^{-1}h_1B_t=B_t(h_0^{-1}h_1)$, $t = d(g_0, g_1) < \that$. Since $h_0^{-1}h_1 \neq \Id$, then $B_t \cap B_t(h_0^{-1}h_1) = \emptyset$, thus $h_0^{-1}h_1g_1 \notin B_t$. So $d(g_0, h_0^{-1}h_1g_1) > t$, i.e., $d(g_0, h_0^{-1}h_1g_1) > d(g_0, g_1)$, which contradicts to \eq{d'mind}.
Thus $h_0 = h_1$,   and the claim follows by the previous paragraph.
\end{proof}

\begin{lemma}\label{lem:tcutth1}
Let $g(\cdot)$ be a \sR geodesic in $G$ starting at $g_0=\Id$ such that $\tcut(g(\cdot))\geq \that$. Then $\tcut(g'(\cdot))\geq \that$ as well.
\end{lemma}
\begin{proof}
Let $\tcut(g(\cdot))\geq \that$. Take any $t \in (0, \that)$. The geodesic $\restr{g(\cdot)}{[0, \that]}$ is optimal, thus $d(g_0, g(t)) = t < \that$. By Lemma \ref{lem:dd'1} we have $d'(g_0', g'(t)) = d(g_0, g(t)) = t$, i.e., the geodesic $\restr{g'(\cdot)}{[0, t]}$ is optimal. Thus $\tcut(g'(\cdot))\geq t$. Taking $t$ arbitrarily close to $\that$, we get the required bound  $\tcut(g'(\cdot))\geq \that$.
\end{proof}

\begin{lemma}\label{lem:tcutth2}
Let $g(\cdot)$ be a \sR geodesic in $G$ starting at $g_0=\Id$ such that $\tcut(g(\cdot))< \that$. Then $\tcut(g'(\cdot)) = \tcut(g(\cdot))$.
\end{lemma}
\begin{proof}
Let $\tcut(g(\cdot))< \that$. Take any $t \in (0, \tcut(g(\cdot)))$. 
The geodesic $\restr{g(\cdot)}{[0, t]}$ is optimal, thus $d(g_0, g(t)) = t < \that$.
By Lemma \ref{lem:dd'1} we have $d'(g_0', g'(t)) = d(g_0, g(t)) = t$, i.e., the geodesic $\restr{g'(\cdot)}{[0, t]}$ is optimal, so $\tcut(g'(\cdot))\geq t$.
Taking~$t$ arbitrarily close to $\tcut(g(\cdot))$, we get  $\tcut(g'(\cdot))\geq \tcut(g(\cdot))$.

Take any $\tau \in (\tcut(g(\cdot)), \that)$.
The geodesic $\restr{g(\cdot)}{[0, \tau]}$ is not optimal, so $d(g_0, g(\tau)) < \tau$, thus $d(g_0, g(\tau)) < \that$.
By Lemma \ref{lem:dd'1} we have 
$d'(g_0', g'(\tau)) = d(g_0, g(\tau)) < \tau$, i.e., the geodesic $\restr{g'(\cdot)}{[0, \tau]}$ is not optimal. Thus $\tcut(g'(\cdot)) = \tcut(g(\cdot))$.
\end{proof}

\begin{lemma}\label{lem:tcutth0}
Let $g(\cdot)$ be a \sR geodesic in $G$ starting at $g_0=\Id$. Then $\tcut(g'(\cdot)) \leq \tcut(g(\cdot))$.
\end{lemma}
\begin{proof}
By contradiction, assume that $\tcut(g'(\cdot)) > \tcut(g(\cdot))$. Take any $t \in (\tcut(g(\cdot)), \tcut(g'(\cdot)))$. 
Then the geodesic $\restr{g'(\cdot)}{[0, t]}$ is optimal, thus its length is equal to $t$: $l\left(\restr{g'(\cdot)}{[0, t]}\right) = t$.
But $l\left(\restr{g(\cdot)}{[0, t]}\right) = l\left(\restr{g'(\cdot)}{[0, t]}\right)$, and the geodesic 
$\restr{g(\cdot)}{[0, t]}$ is not optimal, since  $t > \tcut(g(\cdot))$. Therefore, there exists another geodesic 
from $g(0)$ to $g(t)$ of length less than~$t$. Its projection to $M$ is a geodesic connecting $g'(0)$ and $g'(t)$ of 
the same length less than~$t$. Therefore, the geodesic $\restr{g'(\cdot)}{[0, t]}$ is not optimal, while 
$t<\tcut(g'(\cdot))$. The contradiction thus obtained proves the lemma. 
\end{proof}

Summing up, we have the following bounds of the cut time in $M$.

\begin{corollary}
Let $g(\cdot)$ be a \sR geodesic in $G$ starting at $g_0=\Id$. Then the following bounds hold:
\begin{itemize}
\item[$(1)$]
$\tcut(g'(\cdot)) \leq \tcut(g(\cdot))$.
\item[$(2)$]
If $\tcut(g(\cdot))\geq \that$, then $\that \leq \tcut(g'(\cdot))\leq \tcut(g(\cdot))$.
\item[$(3)$]
If $\tcut(g(\cdot))< \that$, then $\tcut(g'(\cdot))= \tcut(g(\cdot))$.
\end{itemize}
\end{corollary}

Now we compute the number $\that$.

\begin{theorem}\label{th:that}
We have $\that = \frac 12 $.
\end{theorem}

\begin{proof}
In this proof we compute in coordinates $(a,b,c)$.
Denote the intersection $S_h^t = B_t \cap B_t(h)$, $h \in H$, $t > 0$.  Then $S_{(\pm 1, 0, 0)}^{1/2} = \{\left(\pm \frac 12, 0, 0\right)\}$, $S_{(0, \pm 1, 0)}^{1/2} = \{\left(0, \pm \frac 12, 0\right)\}$, and all the rest of the sets  $S_h^{1/2}$ are empty. Thus $\that \leq  \frac 12 $. 

Take any $t \in \left(0, \frac 12 \right)$, then for any $h \in H \setminus \{\Id\}$ we have $d(\Id, S_h^t) > \frac 12$, 
by the above statements and since the points $(\pm\frac12,0,0)$, $(0,\pm\frac12,0)$ are on distance 
$\frac12$ from the identity $\Id=(0,0,0)$. Thus $S_h^t = \emptyset$. So $\that = \frac 12 $.
\end{proof}

We plot the  balls $B_{1/2}(h) \subset G$ for $h \in \{(0, 0, 0), (\pm 1, 0, 0), (0, \pm 1, 0)\}$ in Fig.~\ref{fig:that}, and the ball $B'_{1/2} \subset M$ in Fig.~\ref{fig:thatM}.

\onefiglabelsizen{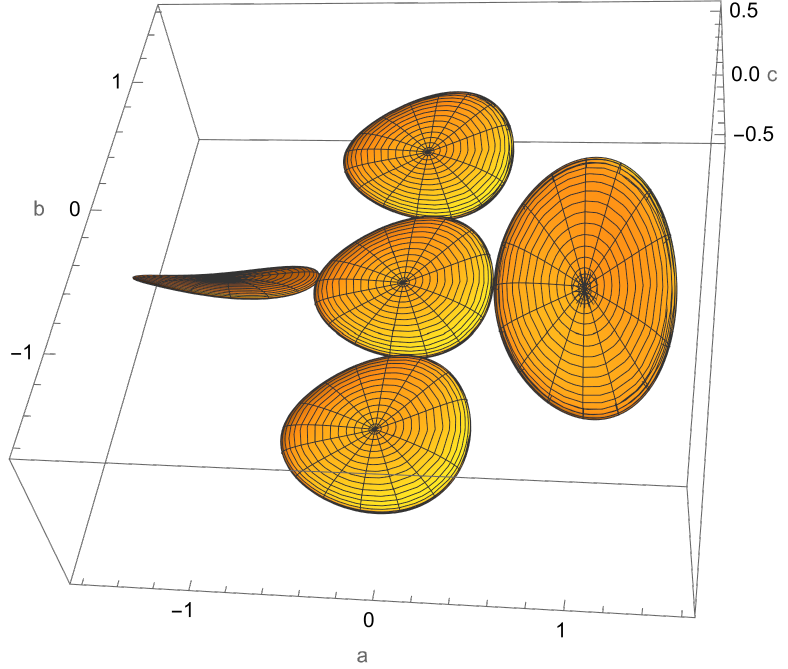}{Sub-Riemannian balls $B_{1/2}(h) \subset G$ touching one another}{fig:that}{7.5}
\onefiglabelsizen{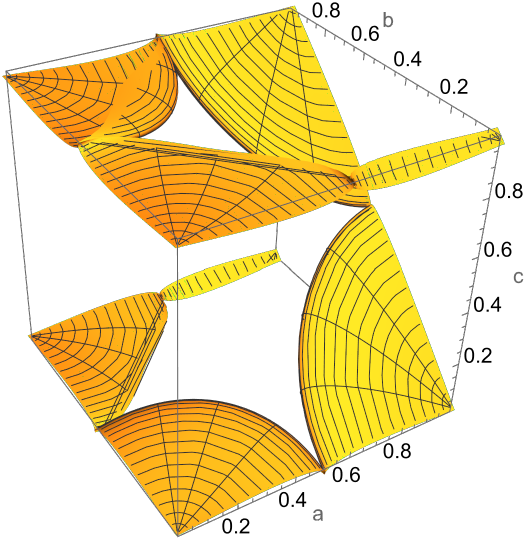}{Sub-Riemannian ball $B'_{1/2}\subset M$ touching itself}{fig:thatM}{7.5}

\begin{remark}
For the quotient of the Euclidean metric from $\R^n$ to $\T^n = \R^n/\Z^n$ (see Sec. \ref{sec:ex}) we have $\that = \frac 12 $ as well.
\end{remark}

\begin{theorem}\label{th:that1}
We have 
$$
\sup \{t> 0 \mid \d_t(E_2) \cap h \d_t(E_2) = \emptyset \ \forall h \neq \Id \in H\} =  \frac 12.
$$
\end{theorem}

\begin{proof}
Similarly to Theorem \ref{th:that} since the ellipsoid $\d_t(E_2)$ is tangent to the \sR sphere $B_t$ along the equator $B_t \cap \{ z = 0\}$.
\end{proof}

We plot the sets $h \d_{1/2}(E_2)$ for $h \in \{(0, 0, 0), (\pm 1, 0, 0), (0, \pm 1, 0)\}$ in Fig.~\ref{fig:that1}.

\onefiglabelsizen{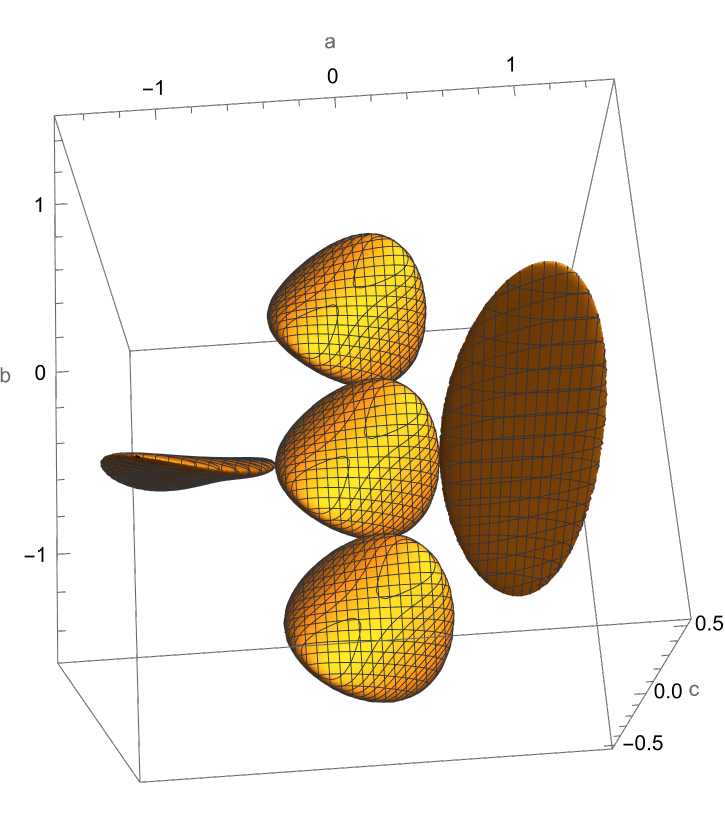}{Sets $h \d_{1/2}(E_2)$ touching one another}{fig:that1}{7.5}

\begin{remark}
Let $q(t)$ be a periodic \sR geodesic in $M$ of period~$T$. Then it is obvious that $\tcut(q(\cdot)) \leq \frac T2$ since $q(T/2)$ is a Maxwell point \cite{intro}, i.e., an intersection point of two symmetric geodesics. 

In a special case this bound turns into equality. Consider a geodesic~$q(t)$ of the form \eq{c'} with $\theta = \frac{\pi n}{2}$, $n \in \Z$. Then it is easy to see that $\tcut(q(\cdot)) = \frac T2 = \frac 12$.
\end{remark}

\section{Conclusion}\label{sec:conclude}

This work seems to be the first study of a projection of a left-invariant \sR structure on a Lie group to a compact homogeneous space. It reveals essential difference between the initial structure and its projection despite their local isometry. 

For example, dynamical behaviour of \sR geodesics on the Heisenberg group $G$ is trivial: all  geodesics tend to infinity. Dynamics of \sR geodesics on the Heisenberg 3D nil-manifold~$M$ includes closed geodesics, dense in $M$ geodesics, and geodesics dense in a 2D torus. 

Further, optimality of \sR geodesics in $G$ is very well understood; the corresponding cut time  arises due to continuous symmetries of the \sR structure on $G$. Description of optimality and cut time on $M$ is much delicate since there are no continuous symmetries; and discrete symmetries which seem to generate the cut locus are hidden since they do not respect the projection mapping $\map{\pi}{G}{M}$. Although, some two-sided bounds of the cut time in~$M$ are possible due to two-sided bounds of \sR balls and distance in $G$, which may be of independent interest.

\section*{Acknowledgement}
The authors thank V.M. Buchstaber, {I.A. Taimanov,} and A.V. Podobryaev for helpful discussions of certain aspects of this work.

\end{document}